\documentclass[11pt,a4paper,leqno]{amsart}

\usepackage{common}

\title{On the compactness of the bi-commutator}
\author{Henri Martikainen}
\author{Tuomas Oikari}

\address[H.M.]{Department of Mathematics and Statistics, Washington University
in St. Louis, 1 Brookings Drive, St. Louis, MO 63130, USA}
\email{henri@wustl.edu}

\address[T.O.]{Departamento de Matem\`atiques, Universitat Aut\`onoma de Barcelona,
	Edifici C Facultat de Ci\`encies, 08193 Bellaterra (Barcelona), Catalonia}
\email{tuomas.oikari@gmail.com}

\makeatletter
\@namedef{subjclassname@2010}{%
    \textup{2010} Mathematics Subject Classification}
\makeatother


\begin{document}

\allowdisplaybreaks

\begin{abstract}
    We prove compactness results and characterizations for the bi-commutator
    $[T_1, [b, T_2]]$ of a symbol $b$ and two non-degenerate
    Calder\'on--Zygmund singular integral operators
    $T_1, T_2$. Our strategy for proving sufficient conditions for compactness is to first establish them in
    the mixed-norm $L^{p_1}L^{p_2}\to L^{q_1}L^{q_2}$ off-diagonal case with $p_i < q_i$,
    and then extend these to other exponents, including the diagonal $p_i = q_i$,
    with a new extrapolation argument. In particular, the natural product $\VMO$
    condition is obtained as a sufficient condition in the diagonal.
    A full characterization is obtained,
    both in terms of a vanishing mean oscillation type condition and in
    terms of the approximability of the symbol, whenever the inequality $p_i \le q_i$
    is strict for at least one index. The extrapolation scheme for proving sufficiency
    requires us to prove new approximation results in relevant bi-parameter function
    spaces that are of independent interest. The necessity results are obtained
    by carefully combining recent rectangular approximate weak factorization methods with a classical
    idea of Uchiyama.
\end{abstract}

\maketitle

\section{Introduction and main results}
A commutator of pointwise multiplication and Calder\'on--Zygmund operator (CZO)
takes the form $[b,T]f := bTf - T(bf)$. A CZO is an $L^2$ bounded singular integral
operator (SIO) -- formally, an operator of the form
$$
  Tf(x)=\int_{\R^d}K(x,y)f(y)\ud y,
$$
where different assumptions on the {\em kernel} $K$ lead to
various classes of linear transformations arising across analysis.
Both the boundedness and compactness of the above commutator is fully understood.
In the very recent paper \cite{AHLMO2021}
the off-diagonal $L^{p_1}_{x_1}L^{p_2}_{x_2}\to L^{q_1}_{x_1}L^{q_2}_{x_2}$
boundedness theory of bi-commutators $[T_1, [b, T_2]]$ was developed -- here each $T_i$
is a Calder\'on--Zygmund operator on $\R^{d_i}$ and $\R^d$ is viewed as
the bi-parameter product space $\R^d = \R^{d_1} \times \R^{d_2}$. In the current
companion paper we study compactness questions. 
Interestingly, our new methods
are such that they make full use of the off-diagonal bounds -- even in the diagonal case --
thus showcasing, beyond their intrinsic interest, also the usefulness of such general estimates. 
In particular, we obtain, via a new off-diagonal extrapolation method for bi-commutators,
a sufficient condition for compactness also in the diagonal $L^p \to L^p$ case in terms
of a product $\VMO$ condition, see Lacey--Wick \cite{LaTeWi06} 
for the related Hilbert case.

To understand the context and the appearing function spaces, we first recap
the one-parameter theory. The results of Nehari \cite{Nehari1957}
and Coifman--Rochberg--Weiss \cite{CRW} characterize the diagonal $L^p \to L^p$
boundedness of the Hilbert and Riesz commutators, respectively, in terms
of $\BMO$ -- the usual space of functions of bounded mean oscillation:
$$
\Norm{b}{\BMO} :=\sup_I \fint_I  \abs{b-\ave{b}_I},
$$
where the supremum is over all cubes $I \subset \R^d$ and $\ave{b}_I = \fint_I b := \frac{1}{|I|} \int_I b$.
The off-diagonal situation $[b,T]\colon L^p\to L^q$, $p \ne q$ is also understood:
for $1<p<q<\infty$ the $\BMO$ space is replaced by the homogeneous H\"older space $\dot C^{0,\beta}$,
$$
\|[b,T]\|_{L^p \to L^q} \sim \Norm{b}{\dot C^{0,\beta}} 
:= \sup_{x \ne y} \frac{|b(x) - b(y)|}{|x-y|^{\beta}}, \qquad \beta
:=d\Big(\frac{1}{p}-\frac{1}{q}\Big),
$$
see Janson \cite{Jan1978}. The boundedness in the remaining range $1<q<p<\infty$
was characterised only very recently by Hyt\"onen \cite{HyLpq2021}:
$$
 \|[b,T]\|_{L^p \to L^q} \sim \Norm{b}{\dot L^{r}}
 := \inf_{c \in \C} \Norm{b-c}{L^r}, \qquad \frac{1}{r} :=\frac{1}{q}-\frac{1}{p}.
$$
These results hold for very general CZOs -- the lower bounds, or necessary conditions,
require only very mild non-degeneracy of the kernels.
The result in this last range is somewhat surprising -- it says that there is
essentially no cancellation between $bT$ and $Tb$ in this regime (indeed, the upper
bound is a triviality for both terms individually). Yet, it was shown in \cite{HyLpq2021} to have
applications related to a conjecture of Iwaniec \cite{Iwa1997} concerning
the prescribed Jacobian problem.

Regarding the multi-parameter case, the result of \cite{AHLMO2021} is as follows.
Given  $p_i, q_i \in (1, \infty),$ let $\beta_i,r_i$ be defined through the relations
\begin{equation*}
\begin{split}
  \beta_i := d_i\left( \frac{1}{p_i}-\frac{1}{q_i}\right), \quad \text{if}\quad p_i<q_i;\qquad
  \frac{1}{q_i} :=\frac{1}{r_i}+\frac{1}{p_i},\quad \text{if}\quad p_i>q_i.
\end{split}
\end{equation*}
Let $T_1$ and $T_2$ be two (symmetrically) non-degenerate CZOs on $\R^{d_1}$ and $\R^{d_2}$,
respectively, and $b \in L^2_{\loc}(\R^{d})$, $d=d_1+d_2$. 
Then $\|[T_1, [b, T_2]]\|_{L^{p_1}L^{p_2}\to L^{q_1}L^{q_2}}$
has upper and lower bounds according to the following table:
\begin{center}
	\begin{tabular}{| c | c | c | c |}
		\hline
		$\vec{p}, \vec{q}$  & $p_1<q_1$ & $p_1=q_1$ & $p_1>q_1$ \\ 
		\hline
		& & & $\lesssim\|b\|_{\dot L^{r_1}(\dot C^{0,\beta_2})}$ \\
		$p_2<q_2$
        & $\sim\|b\|_{\dot C^{0,\beta_1}(\dot C^{0,\beta_2})}$
        & $\sim\|b\|_{\dot C^{0,\beta_2}(\BMO)}$ 
        &  $\gtrsim\|b\|_{\dot C^{0,\beta_2}(\dot L^{r_1})}$ \\ 
		\hline
		&
		& $\lesssim \|b\|_{\BMO_{\rm prod}}$
		&  \\
		$p_2=q_2$ 
        &  $\sim\|b\|_{\dot C^{0,\beta_1}(\BMO)}$ 
        & $\gtrsim \|b\|_{\BMO_{\operatorname{rect}}}$ 
        &  $\lesssim\|b\|_{\dot L^{r_1}(\BMO)}$ \\ 
		\hline
		$p_2>q_2$ & $\sim\| b\|_{\dot C^{0,\beta_1}(\dot L^{r_2})}$
		& $\lesssim\| b\|_{\BMO(\dot L^{r_2})}$
		& $\lesssim\| b\|_{\dot L^{r_1}(\dot L^{r_2})}$	\\
		\hline
	\end{tabular}
\end{center}
Only the lower bounds require the SIOs (or rather the kernels) to be non-degenerate,
and only the upper bounds require the a priori $L^2$ boundedness of the underlying SIOs (that is CZOs).

The off-diagonal cases that do not naturally involve
the space $\dot L^r$ are all satisfactory (indicated by $\sim$ in the table),
while most of the other cases are not yet fully resolved unlike in the one-parameter
case \cite{HyLpq2021}. The diagonal is special and of special interest as well. The product $\BMO$ upper
bound in this generality of arbitrary CZOs was proved in \cite{DaO} (previous special
cases were known, notably \cite{FerSad}) -- the rectangular
$\BMO$ lower bound recorded above for general non-degenerate CZOs does not match this.
The diagonal lower bound in terms of product $\BMO$ 
for special singular integrals (Hilbert/Riesz transforms)
has a somewhat complicated history 
and is considered especially difficult. While the product $\BMO$ of Chang and Fefferman
is still believed to be the characterizing condition, the existing proofs 
have been reported to have a gap,
see e.g. \cites{Ferguson2002, HTV2021ce, AHLMO2021}, and to the best of our knowledge
a fix has not yet been published. Due to this, and also due to the fact that
we deal with completely general non-degenerate CZOs, our aim in the diagonal
is simply to provide sufficiency in terms of a natural product $\VMO$ space.
However, we do characterize compactness in multiple off-diagonal cases -- in fact,
we use those characterizations to derive the diagonal case. Before saying more, let us discuss
compactness in the one-parameter situation.

Let $X^{p, q}$ denote the space that characterizes
the boundedness $[b,T] \colon L^p \to L^q$ for
given exponents $p,q \in (1, \infty)$. That is, let $X^{p,p} := \BMO$,
$X^{p, q} := \dot C^{0, \beta}$ for $p < q$, and $X^{p,q} := \dot L^r$ for $p > q$
using the notation from above. Then, compactness is characterized
by the space
$$
    Y^{p, q} := \overline{C^{\infty}_c}^{X^{p, q}}.
$$
The case $p = q$ is by Uchiyama \cite{Uch1978}, the case $p < q$
by Guo, He, Wu and Yang \cite{GHWY21} and the case $p > q$ by Hyt\"onen, Li, Tao and Yang
\cite{HLTY2023}. Notice that for $p > q$ this means that compactness
and boundedness are both equivalent with $b \in \dot L^r$, which is again
somewhat curious. There are also some nuances in the way this should be stated when $\beta \ge 1$ --
however, it turns out even $\beta = 1$ leads to constants functions \cite{GHWY21},
and hence the above compactness characterizations in terms of $Y^{p,q}$ in fact only concerns
the case $\beta < 1$. Now, the way we have formulated these characterizations 
is in terms of the approximability of the symbol.  It is non-trivial
to characterize the spaces $Y^{p, q}$ in terms of vanishing mean oscillation ($\VMO$)
type conditions, or some vanishing H\"older type pointwise conditions.
In particular, their bi-parameter versions, are important to us. For us, all of these
characterizations are equally important
-- approximability is needed in our
sufficiency proofs via an extrapolation argument, while the VMO type conditions,
due to their oscillatory characterizations, appear in the necessity proofs
in connection with approximate weak factorization methods.

Before discussing additional details of our method, we state our main theorem.
First, we deliberately limit the scope of this paper to exponents corresponding with the $2 \times 2$
square in the upper left corner of the boundedness table above. That is,
we restrict to situations not involving $\dot L^r$ type spaces, which, if involved, usually
require slightly different methods. 
\begin{thm}\label{thm:main} Let $b\in L^2_{\loc}(\R^d)$ and $T_i$
    be non-degenerate CZOs on $\R^{d_i}$, $i=1,2$. Then, the compactness of the bi-commutator 
	$$
	[T_1,[b,T_2]]\in\calK(L^{p_1,p_2},L^{q_1,q_2}),\qquad p_1,p_2,q_1,q_2\in (1,\infty)
	$$
	has sufficient and necessary conditions according to the following table 
\begin{center}
	\begin{tabular}{ |c | c | c|  }
		\hline
		$\vec{p},\vec{q}$& $p_1<q_1$ & $p_1=q_1$  \\ 
		\hline 
			& &  \\
		$p_2 < q_2$ &   iff $\VMO^{\beta_1,\beta_2}$
		&    iff $\VMO^{0,\beta_2}$ \\ 
		\hline
		& &  suff. $\VMO_{\op{prod}}$  \\
	$p_2=q_2$ 	&   iff $\VMO^{\beta_1,0}$
		& nec. $\VMO_{\op{rect}}$ \\ 
		\hline
	\end{tabular}
\end{center}
\end{thm}
The appearing spaces are defined in full detail in the following introductory sections.
Except in the diagonal, we have multiple different equivalent ways to define them --
all of them important to our methods. In the diagonal $\VMO_{\textup{prod}}$ is defined as the closure
of $C^{\infty}_c$ in the product $\BMO$ norm of Chang--Fefferman -- this
corresponds with Lacey--Wick \cite{LaTeWi06} dealing with the Hilbert case.

Our sufficiency proof works by proving directly the case where $p_1 < q_1$ and $p_2 < q_2$.
Then the other cases are, for the lack of a better word, extrapolated from this range. 
In particular, we critically utilize off-diagonal results to prove diagonal results,
further motivating the results of this paper and \cite{AHLMO2021}. For this method,
the approximability viewpoint is critical. In turn, this means that we will
need to prove results, such as, $\VMO^{\beta_1,0}(\R^{d}) =
\overline{C^{\infty}_c(\R^d)}^{\BMO^{\beta_1,0}(\R^d)}$,
where the appearing $\VMO$ and $\BMO$ spaces are defined using
natural rectangular oscillations. Such density results
are delicate and the new bi-parameter proofs rely, among other things, on recent results
by Mudarra and the second name author \cite{MudOik24}, and constructions by Uchiyama \cite{Uch1978}.

Necessary conditions are proved differently. They are based on recent bi-parameter rectangular
weak factorization methods proved in \cite{AHLMO2021}. Those are non-trivial
generalizations of the original one-parameter method \cite{HyLpq2021}, where
oscillatory characterizations of the appearing functions spaces are dualized, and the dualizing
functions are decomposed in a way that generate the bi-commutator together
with some error terms. We combine these type of arguments, where we highlight that
it is critical to start from the oscillatory characterization (as opposed to an approximability
characterization), with a version of an algorithm of Uchiyama \cite{Uch1978} that is
used to obtain that compactness cannot hold for symbols outside of our
$\VMO$ type spaces.

\subsection*{Acknowledgements}
This material is based upon work supported by the National Science Foundation
under Grant No. 2247234 (H. Martikainen). H.M. was, in addition, supported by the
Simons Foundation through MPS-TSM-00002361 (travel support for mathematicians).
T.O. was supported by the Finnish Academy of Science and Letters, and
by the MICINN (Spain) grant no. PID2020-114167GB-I00.

\subsection{Notation and definitions}

\begin{itemize}
\item We identify $f \colon \R^{d_1}\times\R^{d_2} \to \C$ with the family of functions
$f_{x_1} = f(x_1, \cdot):\R^{d_2}\to\C,$ indexed by $x_1\in\R^{d_1},$ and understand that
$$
\| f\|_{L^{p_1}(\R^{d_1}; L^{p_2}(\R^{d_2}))} = \|  \| f_{x_1}(x_2)  \|_{L^{p_2}_{x_2}} \|_{L^{p_1}_{x_1}}.
$$
We also denote $L^{p_1}(\R^{d_1}; L^{p_2}(\R^{d_2})) = L^{p_1}_{x_1} L^{p_2}_{x_2} = L^{p_1}L^{p_2} = L^{p_1,p_2}.$
Similar natural interpretations hold for other functions spaces that appear in the article.
\item  
Cubes in $\R^{d_i}$ are generally denoted by $I_i$ --
rectangles $R\subset \R^d := \R^{d_1} \times \R^{d_2}$ then take the form $I_1\times I_2.$
Side length, diameter and and centres are denoted by $\ell(I_i),$ $\diam(I_i)$ and $c_{I_i}$,
respectively. In the last Section \ref{sect:NecesaryConditions},
to make additional room for subindices,
we recast $I_1\mapsto I$ and $I_2\mapsto J$ 
-- there rectangles $R\subset \R^d$, thus, take the form $I\times J$.
\item Averages are denoted by
    $$
        \langle f \rangle_A := \fint_A f = \frac{1}{|A|} \int_A f,
    $$
    where $|A|$ is the Lebesgue measure of $A$.
\item 
Various operators, initially defined only in one of the spaces $\R^{d_i}$,
act on the product space $\R^d$ in the natural way.
For instance, $T_1$ is a CZO on $\R^{d_1}$ but it can act on
functions $f$ defined in $\R^d$ via
$$
    T_1f(x) := T_1(f(\cdot,x_2))(x_1).
$$
Similarly, if $I_2 \subset \R^{d_2}$ is a cube, we define the average of $f$ over $I_2$
to be the function
$$
    \ave{f}_{I_2}(x_1) :=  \ave{f(x_1,\cdot)}_{I_2}.
$$
\item Natural vectors of exponent or coefficient pairs may appear throughtout, such as,
$\vec{p} := (p_1,p_2)$ and $\vec{\beta} := (\beta_1,\beta_2)$.
This then leads to natural notation, such as, $\dot{C}^{0,\beta_1,\beta_2} :=
\dot{C}^{0,\vec{\beta}}$.
\item 
We denote $A \lesssim B$, if $A \leq C B$ for some constant $C>0$ depending only
on parameters like integration exponents or the dimension $d$ that are not important to track.
Furthermore, we set $A \sim B$, if $A \lesssim B$ and $B \lesssim  A$.
Subscripts or variables on constants and quantifiers, such as, $C_{a},C(a)$ and
$\lesssim_{a}$, signify their dependence on those subscripts.

\item We call 
    $$
        K_i \colon \R^{d_i} \times \R^{d_i} \setminus \{(x_i,y_i) \in \R^{d_i} 
        \times \R^{d_i} \colon x_i =y_i \} \to \C
    $$
    a standard Calder\'on-Zygmund kernel on $\R^{d_i}$ if we have
    $$
        |K(x_i,y_i)| \le \frac{C}{|x_i-y_i|^{d_i}}
    $$
    and, for some $\alpha_i \in (0,1]$, we have
    \begin{equation}\label{eq:holcon}
        |K(x_i, y_i) - K(x_i', y_i)| + |K(y_i, x_i) - K(y_i, x_i')|
        \le C\frac{|x_i-x_i'|^{\alpha_i}}{|x_i-y_i|^{d_i+\alpha_i}}
\end{equation}
whenever $|x_i-x_i'| \le |x_i-y_i|/2$. 

A singular integral operator (SIO) is a linear operator $T_i$ on $\R^{d_i}$
(initially defined, for example, on bounded and compactly supported functions) so that there is a standard kernel
$K_i$ for which
$$
    \langle T_if ,g \rangle = \iint_{\R^{d_i} \times \R^{d_i}} K_i(x_i,y_i) f(y_i) g(x_i) \ud y_i \ud x_i
$$
whenever the functions $f$ and $g$ are nice and have disjoint supports.
A Calder\'on--Zygmund operator (CZO) is an SIO $T_i$, which is bounded
from $L^p(\R^{d_i}) \to L^p(\R^{d_i})$ for all (equivalently for some) $p \in (1,\infty)$.
The $T1$ theorem  says that an SIO is a CZO if and only if
$$
    \int_{I_i} |T_i1_{I_i}| + \int_{I_i} |T_i^*1_{I_i}| \lesssim |I_i|
$$
for all cubes $I_i \subset \R^{d_i}$. Here $T_i^*$ is the linear adjoint of $T_i$.
 
\end{itemize}

\section{On relevant function spaces and their properties} 

\subsection{Product $BMO/VMO$} 
We begin by defining product $\BMO$ and then product $\VMO$ -- these
are the spaces that are relevant for the boundedness and compactness, respectively,
of bi-commutators in the diagonal $L^p \to L^p$ or, more generally,
in the mixed-norm diagonal $L^{p_1}L^{p_2} \to L^{p_1}L^{p_2}$.
Although, the necessity of these conditions is not strictly speaking known,
see the Introduction.

We define the product $\BMO$ as the supremum of dyadic product $\BMO$ norms as follows:
\[
\|b \|_{\BMO_{\op{prod}}(\R^d)} 
:= \sup_{\calD = \calD^1\times\calD^2} \sup_{\Omega}\Big(\frac{1}{|\Omega|}
\sum_{\substack{ R\in\calD \\ R\subset\Omega}}|\langle b,h_R \rangle|^2\Big)^{1/2},
\]
where the supremum is taken over all dyadic product lattices $\calD$
and over all non-empty open subsets $\Omega \subset \R^d$
of finite measure. Here $h_R$ is a cancellative Haar function in the rectangle
$R=I_1\times I_2$ with $I_i\in\calD^i$. This product $\BMO$ of Chang and Fefferman
is a rather complicated space, but we do not need to work with it via the definition
in this paper. 

We also point out that this space is well-known to satisfy John--Nirenberg in the following form
$$
    \sup_{\Omega}\Big(\frac{1}{|\Omega|}
    \sum_{\substack{ R\in\calD \\ R\subset\Omega}}|a_R|^2\Big)^{1/2}
    \sim
    \sup_{\Omega} \frac{1}{|\Omega|^{1/p}} \Big\| \Big(
    \sum_{\substack{ R \in \calD \\ R \subset \Omega}} |a_R|^2 \frac{1_R}{|R|} \Big)^{1/2} \Big\|_{L^p},
    \qquad 0 < p < \infty.
$$
Thus, there is nothing special about the exponent $2$ in the definition of product $\BMO$.
In fact, this will be the case for most of the spaces appearing in this paper,
and we will often even provide a quick argument showing that.
There is one notable exception -- the basic, non-fractional
in both parameters ($\beta_1 = \beta_2 = 0$), rectangular $\BMO$ (the space $\BMO_{v_1, v_2}^{0,0}$,
see \eqref{eq:rectBMOdefn} below)
does not satisfy John--Nirenberg and thus may depend on the exponents used in the definition.
Essentially, when the space characterizes boundedness of commutators (or in the diagonal
is believed to do so), it will satisfy John--Nirenberg.
This is also the case for the various \emph{fractional} (but rectangular) $\BMO$ spaces
that are the correct spaces in off-diagonal situations.
 It is only in the full diagonal, where the rectangular
version is not the correct space (and does not even satisfy John--Nirenberg)
and the more complicated product $\BMO$ variant involving general open sets $\Omega$ is required.
We will only encounter the rectangular non-fractional $\BMO/\VMO$ (with some exponents, remember that they now matter)
in the diagonal in connection with necessary conditions --
again, these do not match the known sufficient conditions involving the product variants.

We define product VMO in terms of approximability by nice functions. 
\begin{defn} Define 
	\[
	\VMO_{\op{prod}}(\R^d) := \overline{C^{\infty}_c(\R^d)}^{\BMO_{\op{prod}}(\R^d)}.
	\]
\end{defn}
\begin{rem}
	Notice that $C^{\infty}_c\subset \BMO_{\op{prod}}$ so that the definition is reasonable.
    Indeed, one can go through the easier little $\BMO$ space $\operatorname{bmo}$ defined by
    $$
        \|b\|_{\operatorname{bmo}}
        := \sup_{R \subset \R^d \textup{ is a rectangle}} \fint_R |b-\langle b \rangle_R|.
    $$
    This is because
	\[
	    C^{\infty}_c \subset L^{\infty} \subset  \operatorname{bmo} \subset \BMO_{\op{prod}}.
	\]
	The last inclusion is non-trivial but well-known -- for a recent reference see
    \cite[Appendix A]{LMV2019biparBloom}.
\end{rem}

\subsection{H\"older spaces and rectangular fractional $\BMO$/$\VMO$}
For $b \in L^{1}_{\loc}$ and a rectangle $R = I_1 \times I_2$ we denote
\begin{equation*}
    \osc_{v_1,v_2}(b,R)
    :=\frac{\Norm{b-\langle b\rangle_{I_1} -\langle b\rangle_{I_2} 
    		+\langle b\rangle_{R}}{L^{v_1}_{x_1}L^{v_2}_{x_2}(R)}}{|I_1|^{1/v_1} |I_2|^{1/v_2}}, \qquad 1 \le v_i < \infty.
\end{equation*}
If $v_1=v_2=1$ then we abbreviate
$\osc(b,R) := \osc_{1,1}(b,R)$. Mostly, due to John-Nirenberg (see the discussion above),
this choice of exponents is natural.
Given some $\beta_1, \beta_2$ we then also define 
$$
    \calO^{\beta_1, \beta_2}_{v_1, v_2}(b, R) :=  \ell(I_1)^{-\beta_1}\ell(I_2)^{-\beta_2} \osc_{v_1,v_2}(b, R).
$$
If the numbers $\beta_i$ are strictly positive, we think of the spaces below as being fractional.
We define the (fractional in parameter $i$ if $\beta_i > 0$) rectangular $\BMO$ norm
\begin{equation}\label{eq:rectBMOdefn}
    \|b\|_{\BMO^{\beta_1, \beta_2}_{v_1, v_2}} =
    \sup_{R=I_1 \times I_2} \calO^{\beta_1, \beta_2}_{v_1, v_2}(b, R).
\end{equation}

\subsubsection*{Purely fractional case}
We discuss the strictly fractional case, where $\beta_i > 0$ for $i=1,2$.
The bi-parameter homogeneous H\"older norm, for $\beta_i > 0$, is
\begin{align*}
    \|b\|_{\dot C^{0, \beta_1, \beta_2}} &:= \|b\|_{\dot C^{0, \beta_1}_{x_1}(\dot C^{0, \beta_2}_{x_2}) }
    = \|b\|_{\dot C^{0, \beta_2}_{x_2}(\dot C^{0, \beta_1}_{x_1}) }  \\
                                               &= \sup_{x_1 \ne y_1} \Big\|\frac{b(x_1, \cdot)- b(y_1, \cdot)}
                                               {|x_1-y_1|^{\beta_1}} \Big\|_{\dot C^{0, \beta_2}}
                                               = \sup_{x_2 \ne y_2} \Big\|\frac{b(\cdot, x_2)- b(\cdot, y_2)}
                                               {|x_2-y_2|^{\beta_1}} \Big\|_{\dot C^{0, \beta_2}} \\
                                               &= \sup_{\substack{x_1\neq y_1\\ x_2\neq y_2}}
                                               \frac{|b(x_1, x_2)-b(x_1, y_2)-b(y_1,x_2)+b(y_1, y_2)|}
                                               {|x_1-y_1|^{\beta_1}|x_2-y_2|^{\beta_2}}.
\end{align*}
In this purely fractional case, the above defined fractional $\BMO$ norm agrees with the homogeneous H\"older norm.
The result below is essentially \cite[Proposition 3.5]{AHLMO2021} with general $v_i$ instead of $v_i = 1$.
While we mostly care about the corresponding $\VMO$ variants of such results (since we deal with
compactness in this paper), the proof is useful as
variants of it are used later in the analogous $\VMO$ results.
\begin{lem}
    Let $1 \le v_i < \infty$ and $\beta_i > 0$ for $i=1,2$. Then for $b \in L^1_{\loc}$ we have
    $$
         \|b\|_{\dot C^{0, \beta_1, \beta_2}} \sim \|b\|_{\BMO^{\beta_1, \beta_2}_{v_1, v_2}}. 
    $$
\end{lem}
\begin{proof}
    For $R = I_1 \times I_2$ notice that for all $x_i \in I_i$ we have
    \begin{align*}
        \abs{b(x_1,x_2)-&\ave{b}_{I_1}(x_2)-\ave{b}_{I_2}(x_1)+\ave{b}_{I_1\times I_2}} \\
                        &\le \fint_{I_1\times I_2}|b(x_1, x_2)-b(x_1,y_2)-b(y_1,x_2)+b(y_1,y_2)|\ud y_1\ud y_2 \\
                        &\lesssim \|b\|_{\dot C^{0, \beta_1, \beta_2}} \ell(I_1)^{\beta_1} \ell(I_2)^{\beta_2},
    \end{align*}
    and so $\|b\|_{\BMO^{\beta_1, \beta_2}_{v_1, v_2}} \lesssim \|b\|_{\dot C^{0, \beta_1, \beta_2}}.$

    For the converse, fix $x_1 \ne y_1$ and $x_2 \ne y_2$. We define
    \begin{align}\label{eq:expand}
    	   x_{i, k}:=(1-2^{-k})x_i+2^{-k}y_i \qquad \textup{and} \qquad y_{i,k} :=(1-2^{-k})y_i+2^{-k}x_i,
    \end{align}
    and note that $x_{i, 1} = y_{I} =\frac12(x_i+y_i)$.
    If $Q_k(x_i) := Q(x_{i, k}, 2^{-k}\abs{x_i-y_i})$ is the cube centered at $x_{i, k}$ and of side-length
    $2\cdot 2^{-k}\abs{x_i-y_i}$, we easily check that $Q_{k+1}(x_i)\subset Q_k(x_i)$ and
    $Q_k(x_i) \subset B(x_i, C2^{-k}|x_i-y_i|)$. 
    Write $R_{k_1, k_2}(w_1, w_2) := Q_{k_1}(w_1) \times Q_{k_2}(w_2)$ for $w_i \in \{x_i, y_i\}$.
    We have
    \begin{equation}\label{eq:H1}
        \begin{split}
            &b(x_1, x_2)-b(x_1,y_2)-b(y_1,x_2)+b(y_1,y_2) \\
        &=  \sum_{k_1, k_2 =0}^{\infty} \big[ \bave{b}_{R_{k_1+1, k_2+1}(x_1, x_2)} 
            - \bave{b}_{R_{k_1, k_2+1}(x_1, x_2)}
        -\bave{b}_{R_{k_1+1, k_2}(x_1,x_2)}  + \bave{b}_{R_{k_1,k_2}(x_1,x_2)} \big] \\
        &- \sum_{k_1, k_2 =0}^{\infty} \big[ \bave{b}_{R_{k_1+1, k_2+1}(x_1, y_2)} 
            - \bave{b}_{R_{k_1, k_2+1}(x_1, y_2)}
        -\bave{b}_{R_{k_1+1, k_2}(x_1,y_2)}  + \bave{b}_{R_{k_1,k_2}(x_1,y_2)} \big] \\
        &- \sum_{k_1, k_2 =0}^{\infty} \big[ \bave{b}_{R_{k_1+1, k_2+1}(y_1, x_2)} 
            - \bave{b}_{R_{k_1, k_2+1}(y_1, x_2)}
        -\bave{b}_{R_{k_1+1, k_2}(y_1,x_2)}  + \bave{b}_{R_{k_1,k_2}(y_1,x_2)} \big] \\
        &+ \sum_{k_1, k_2 =0}^{\infty} \big[ \bave{b}_{R_{k_1+1, k_2+1}(y_1, y_2)} 
            - \bave{b}_{R_{k_1, k_2+1}(y_1, y_2)}
        -\bave{b}_{R_{k_1+1, k_2}(y_1,y_2)}  + \bave{b}_{R_{k_1,k_2}(y_1,y_2)} \big].
        \end{split}
    \end{equation}
    Consider, for instance, the first term from above -- it can be dominated by
    \begin{align*}
        \sum_{k_1, k_2 =0}^{\infty}&  \fint_{R_{k_1, k_2}(x_1, x_2)}\abs{ b-\ave{b}_{Q_{k_1}(x_1),1}-
        \ave{b}_{Q_{k_2}(x_2), 2} +\ave{b}_{R_{k_1, k_2}(x_1,x_2)}} \\
                                   &\lesssim \sum_{k_1, k_2 =0}^{\infty} (2^{-k_1} |x_1-y_1|)^{\beta_1}
                                   (2^{-k_2}|x_2-y_2|)^{\beta_2} \calO^{\beta_1, \beta_2}_{1, 1}(b, R_{k_1,k_2}(x_1,x_2)) \\
                                   &\lesssim |x_1-y_1|^{\beta_1} |x_2-y_2|^{\beta_2}
                                   \|b\|_{\BMO^{\beta_1, \beta_2}_{1, 1}} 
                                   \le  |x_1-y_1|^{\beta_1} |x_2-y_2|^{\beta_2} \|b\|_{\BMO^{\beta_1, \beta_2}_{v_1, v_2}}. 
    \end{align*}
\end{proof}
In particular, if $\beta_1, \beta_2 > 0$, the choice of the exponents $v_i \ge 1$ does not matter, and we might as well
use $v_1 = v_2 = 1$ and denote
$$
    \BMO^{\beta_1, \beta_2} =  \BMO^{\beta_1, \beta_2}_{1, 1}.
$$
Since we consider compactness in this paper, we are mainly interested in the corresponding $\VMO$ spaces,
and we discuss them next. 

\begin{defn}
    Let $1 \le v_i < \infty$ and $\beta_i \ge 0$. 
    We say $b \in \BMO^{\beta_1, \beta_2}_{v_1, v_2}$ belongs to $\VMO^{\beta_1, \beta_2}_{v_1, v_2}$ if
    for all $\varepsilon > 0$ there exists $t > 0$ so that
    $$
    \calO^{\beta_1, \beta_2}_{v_1, v_2}(b, R) < \varepsilon, \qquad R = I_1 \times I_2,
    $$
    whenever $\ell(I_i) < t$ or $\ell(I_i) > t^{-1}$ for $i=1$ or $i=2$, or
    $\dist(R, \{0\}) > t^{-1}$.
\end{defn}
\begin{defn}\label{def:VCbipar}
    Given $\vec\beta = (\beta_1,\beta_2),$ with $\beta_i > 0$, we say $b \in  \dot C^{0, \beta_1, \beta_2}$ belongs to $\dot{\op{VC}}^{0, \beta_1, \beta_2}$
    if for all $\varepsilon > 0$
    there exists $t > 0$ so that 
    $$
        |b(x_1, x_2) - b(x_1, y_2) - b(y_1, x_2) + b(y_1, y_2)| \le
        \varepsilon |x_1-y_1|^{\beta_1} |x_2-y_2|^{\beta_2},
    $$
    whenever $|x_i-y_i| < t$ or $|x_i| > t^{-1}$
    or $|y_i| > t^{-1}$ for $i = 1$ or $i=2$.
\end{defn}

\begin{lem}\label{lem:VMO_VC_pure_frac}
    Let $1 \le v_i < \infty$ and $\beta_i  > 0$. Then, we have $\VMO^{\beta_1, \beta_2}_{v_1, v_2} = \dot{\op{VC}}^{0, \beta_1, \beta_2}$.
\end{lem}
\begin{proof}
    Suppose that $b \in \VMO^{\beta_1, \beta_2}_{v_1, v_2}$. In particular, we have
    $b \in\BMO^{\beta_1, \beta_2}$ implying $b \in \dot C^{0, \beta_1, \beta_2}$.
    To prove that $b \in \dot{\op{VC}}^{0, \beta_1, \beta_2}$ we write the decomposition
    \eqref{eq:H1} and estimate again, e.g., the first term arriving at the need to estimate
    $$
        I := \sum_{k_1, k_2 =0}^{\infty} (2^{-k_1} |x_1-y_1|)^{\beta_1}
                                   (2^{-k_2}|x_2-y_2|)^{\beta_2} \calO^{\beta_1, \beta_2}_{v_1, v_2}(b, R_{k_1,k_2}(x_1,x_2)). 
    $$
    Fix $\varepsilon > 0$ and choose $t > 0$ so that
    $$
        \calO^{\beta_1, \beta_2}_{v_1, v_2}(b, R) < \varepsilon, \qquad R = I_1 \times I_2,
    $$
    whenever $\ell(I_i) < t$ or $\ell(I_i) > t^{-1}$ for $i=1$ or $i=2$, or
    $\dist(R, \{0\}) > t^{-1}$. In particular, if we have $|x_1-y_1| < t/2$,
    then we have $\ell(Q_{k_1}(x_1)) = 2 \cdot 2^{-k_1}|x_1-y_1| \le
    2|x_1-y_1| < t$, and so
    $$
        I \le \varepsilon \sum_{k_1, k_2 =0}^{\infty} (2^{-k_1} |x_1-y_1|)^{\beta_1} (2^{-k_2}|x_2-y_2|)^{\beta_2}
            \lesssim \varepsilon |x_1-y_1|^{\beta_1} |x_2-y_2|^{\beta_2}.
    $$
    Similarly, the same holds if $|x_2-y_2| < t/2$.

    Choose some $t_0 < t$ -- in what follows it will eventually be useful to be able to choose $t_0$
    much smaller than $t$.
    Suppose then $|x_1| \ge Ct_0^{-1}$. We consider the cases $|x_1-y_1| < t_0^{-1}$
    and $|x_1-y_1| \ge t_0^{-1}$ separately. Suppose first $|x_1-y_1| < t_0^{-1}$.
    In this case
    \begin{align*} 
        \dist(R_{k_1,k_2}(x_1,x_2), \{0\}) &\ge \dist(Q_{k_1}(x_1), \{0\}) \\
                                           &\ge \dist\Big( Q\Big(\frac12(x_1+y_1), \frac12|x_1-y_1| \Big), \{0\}\Big),
    \end{align*}
    where we notice that
    $$
    |(x_1+y_1)/2| \ge |x_1| - |x_1-y_1|/2 \ge \frac{C}{2}t_0^{-1}
    $$
    and so, if $C$ is chosen suitably, 
    $$
    \dist(R_{k_1,k_2}(x_1,x_2), \{0\}) > t_0^{-1} > t^{-1}
    $$
    for all $k_1$. So in this case again
    $$
        I \le \varepsilon \sum_{k_1, k_2 =0}^{\infty} (2^{-k_1} |x_1-y_1|)^{\beta_1} (2^{-k_2}|x_2-y_2|)^{\beta_2}
            \lesssim \varepsilon |x_1-y_1|^{\beta_1} |x_2-y_2|^{\beta_2}.
    $$
    We then consider the case $|x_1-y_1| \ge t_0^{-1}$ and split
    \begin{align*}
        I &= I_1 + I_2, \\
        I _1&:= \sum_{k_1:\, \ell(Q_{k_1}(x_1)) > t^{-1}}  \sum_{k_2 =0}^{\infty} (2^{-k_1} |x_1-y_1|)^{\beta_1}
                                   (2^{-k_2}|x_2-y_2|)^{\beta_2} \calO^{\beta_1, \beta_2}_{v_1, v_2}(b, R_{k_1,k_2}(x_1,x_2)), \\
        I_2 &:= \sum_{k_1:\, \ell(Q_{k_1}(x_1)) \le t^{-1}}  \sum_{k_2 =0}^{\infty} (2^{-k_1} |x_1-y_1|)^{\beta_1}
                                   (2^{-k_2}|x_2-y_2|)^{\beta_2} \calO^{\beta_1, \beta_2}_{v_1, v_2}(b, R_{k_1,k_2}(x_1,x_2)).
    \end{align*}
    It is obvious that similarly as above $I_1 \lesssim \varepsilon |x_1-y_1|^{\beta_1} |x_2-y_2|^{\beta_2}$.
    In $I_2$ we simply estimate $\calO^{\beta_1, \beta_2}_{v_1, v_2}(b, R_{k_1,k_2}(x_1,x_2))
    \lesssim \|b\|_{\BMO^{\beta_1, \beta_2}} \lesssim_b 1$. Then notice that
    $$
        2 \cdot 2^{-k_1}t_0^{-1} \le 2\cdot2^{-k_1}|x_1-y_1| = \ell(Q_{k_1}(x_1)) \le t^{-1}
    $$
    implies $2^{-k_1} < t^{-1}t_0.$
    It follows that 
    \begin{align*}
        I_2 &\lesssim_b |x_1-y_1|^{\beta_1} |x_2-y_2|^{\beta_2} \sum_{k_1:\,2^{-k_1} < t^{-1}t_0} 2^{-k_1\beta_1} \\
            &\lesssim (t^{-1}t_0)^{\beta_1} |x_1-y_1|^{\beta_1} |x_2-y_2|^{\beta_2}
            \le \varepsilon  |x_1-y_1|^{\beta_1} |x_2-y_2|^{\beta_2} 
    \end{align*}
    provided that $t_0 =  t_0(t,\varepsilon)$ is small enough. Thus $b$ satisfies
    the desired pointwise condition, since the same proof works also if $|y_1|$ is big.
    The cases when $|x_2|$ or $|y_2|$ is big is symmetric.

    We then assume $b \in \dot{\op{VC}}^{0, \beta_1, \beta_2}$.
    Let now $R=I_1\times I_2$. We estimate
    \begin{equation}\label{eq:Bsplit}
        \begin{split}
            \osc_{v_1,v_2}(b, R)
        &\le \Big(\fint_{I_1} \Big( \fint_{I_2} \Big( \fint_R |B(x,y)| \ud y \Big)^{v_2} 
        \ud x_2 \Big)^{\frac{v_1}{v_2}} \ud x_1 \Big)^{\frac{1}{v_1}}, \\
            B(x,y) &:= b(x_1, x_2)-b(x_1, y_2)-b(y_1,x_2)+b(y_1, y_2).
        \end{split}
    \end{equation}
    Fix $\varepsilon > 0$ and choose $t>0$ so that 
    $$
        |B(x,y)| \le  \varepsilon |x_1-y_1|^{\beta_1} |x_2-y_2|^{\beta_2}
    $$
    whenever $|x_i-y_i| < t$ or $|x_i| > t^{-1}$
    or $|y_i| > t^{-1}$ for $i = 1$ or $i=2$.

    If $\ell(I_i) < t$, then for $x,y \in R$ we have $|x_i-y_i| < t$, and so
    $$
        \osc_{v_1,v_2}(b, R)
        \le \varepsilon \ell(I_1)^{\beta_1} \ell(I_2)^{\beta_2}.
    $$
    The same estimate also clearly holds if $\dist(R, \{0\}) > Ct^{-1}$. Let then
    $t_0 < t$ and assume $\ell(I_1) > t_0^{-1}$. We split the 
    $I_1$ integral on the right hand side of \eqref{eq:Bsplit} to
    $I_1 \cap \{|x_1| > t^{-1}\}$ and $I_1 \cap \{|x_1| \le t^{-1} \}$
    giving us two terms, call them $A$ and $B$, respectively.
    Obviously, we again have $A \le \varepsilon \ell(I)^{\beta_1} \ell(I_2)^{\beta_2}$.
    We then estimate
    $$
        B \le \|b\|_{\dot C^{0, \beta_1, \beta_2}} \ell(I_1)^{\beta_1} \ell(I_2)^{\beta_2}
        \Big(\frac{|B(0, t^{-1})|}{|I_1|} \Big)^{\frac{1}{v_1}}.
    $$
    But here
    $$
    \frac{|B(0, t^{-1})|}{|I_1|} \lesssim t_0^{d_1}t^{-d_1}
    $$
    so that $B \le \varepsilon \ell(I_1)^{\beta_1} \ell(I_2)^{\beta_2}$
    if $t_0 = t_0(t)$ is chosen small enough. This gives us that
    $\calO^{\beta_1, \beta_2}_{v_1, v_2}(b, R)$ is small in all of the desired cases.
\end{proof}
If $\beta_i > 0,$ it again makes sense to set
$$
    \VMO^{\beta_1, \beta_2} :=  \VMO^{\beta_1, \beta_2}_{1,1}, 
$$
since, as we saw, the choice of the exponents does not matter.

There is an alternative way to think about $\VC$ in terms of certain uniform inclusions
to the corresponding one-parameter spaces. In this case, it is essentially just a restatement of the definition
but this perspective will be useful later in bi-parameter approximation results. Motivated by this we define.
\begin{defn}
    We write
    $G \subset_u \VC^{0,\alpha}(\R^n)$ (with $``u"$ for uniformly) provided that 
	for all $\varepsilon>0,$ there exists $t>0$ such that if $|x-y|<t$ or $|x|>t^{-1}$ or $|y|>t^{-1},$ then 
	\[
	\sup_{g\in G} \frac{|g(x)-g(y)|}{|x-y|^{\alpha}}< \varepsilon.
	\]
\end{defn}
Now, given $\vec{\beta} = (\beta_1,\beta_2)$ and a function $\psi$ defined on $\R^d = \R^{d_1}\times \R^{d_2},$ we set
\[
\psi_{x_1,y_1}^{\beta_1}(z_2):= \frac{\psi(x_1,z_2)-\psi(y_1,z_2)}{|x_1-y_1|^{\beta_1}},\qquad 
\psi_{x_2,y_2}^{\beta_2}(z_1) := \frac{\psi(z_1,x_2)-\psi(z_1,y_2)}{|x_2-y_2|^{\beta_2}}.
\]
We can rewrite Definition \ref{def:VCbipar} as the following lemma.
\begin{lem}\label{lem:biVCuniform} There holds that $f\in \VC^{0,\vec{\beta}}(\R^{d_1}\times\R^{d_2})$
	iff for both $i=1,2$ there holds that 
	\begin{align}\label{eq:biVCequiv}
		\left\{  f_{x_i,y_i}^{\beta_i}  \right\} _{\substack{x_i,y_i \in\R^{d_i} 
				\\ x_i\not= y_i}} \subset_u \VC^{0,\beta_j}(\R^{d_j}), \qquad j \ne i.
	\end{align}
\end{lem}

\subsubsection*{Partly fractional case}
We now move to the partly fractional case, where $\beta_1 > 0$ but $\beta_2 = 0$ (or the other
way around).
\begin{thm}\label{thm:CBMO}
    Let $1 \le v_1, v_2 < \infty$ and $\beta_1 > 0$. Then we have
    \begin{align*}
         \Norm{b}{\dot C^{0,\beta_1}_{x_1}(\BMO_{x_2})}
         = \sup_{x_1\neq y_1} \left\| \frac{b(x_1, \cdot)-b(y_1, \cdot)}{|x_1-y_1|^{\beta_1}}\right\|_{\BMO(\R^{d_2})}
         \sim \|b\|_{\BMO^{\beta_1, 0}_{v_1, v_2}}.
    \end{align*}
\end{thm}
\begin{proof}
    Let $b_{x_1, y_1} := b(x_1,\cdot)-b(y_1,\cdot)$.
    For $R = I_1 \times I_2$ and $x_i \in I_i$ we have
    \begin{align*}
        \abs{b(x_1, x_2)-&\ave{b}_{I_1}(x_2)-\ave{b}_{I_2}(x_1)+\ave{b}_{I_1\times I_2}} \\ 
                         &\lesssim \ell(I_1)^{\beta_1}\fint_{I_1}  \Bigg|\frac{b_{x_1, y_1}(x_2)}{|x_1-y_1|^{\beta_1}}
                         -\Big \langle \frac{b_{x_1, y_1}}{|x_1-y_1|^{\beta_1}}\Big \rangle_{I_2}\Bigg|\ud y_1.
    \end{align*}
    This gives
    \begin{equation*}
        \begin{split}
            &\osc_{v_1,v_2}(b, R) \\
        &\lesssim \ell(I_1)^{\beta_1} \Big(\fint_{I_1} \Big( \fint_{I_2} \Big(
                    \fint_{I_1} \Big|\frac{b_{x_1, y_1}(x_2)}{|x_1-y_1|^{\beta_1}}
                         -\Big \langle \frac{b_{x_1, y_1}}{|x_1-y_1|^{\beta_1}}\Big \rangle_{I_2}\Big|\ud y_1\Big)^{v_2} 
        \ud x_2 \Big)^{\frac{v_1}{v_2}} \ud x_1 \Big)^{\frac{1}{v_1}} \\
        &\le   \ell(I_1)^{\beta_1} \Big(\fint_{I_1} \Big( \fint_{I_1} \Big(
                    \fint_{I_2} \Big|\frac{b_{x_1, y_1}(x_2)}{|x_1-y_1|^{\beta_1}}
                -\Big \langle \frac{b_{x_1, y_1}}{|x_1-y_1|^{\beta_1}}\Big \rangle_{I_2}\Big|^{v_2}\ud x_2\Big)^{\frac{1}{v_2}} 
    \ud y_1\Big)^{v_1} \ud x_1 \Big)^{\frac{1}{v_1}} \\
        &\lesssim \ell(I_1)^{\beta_1} \Norm{b}{\dot C^{0,\beta_1}_{x_1}(\BMO_{x_2})},
        \end{split}
    \end{equation*}
    where in the penultimate step we used Minkowski's integral inequality and in the last
    step the John--Nirenberg inequality. It follows that
    $ \calO^{\beta_1, 0}_{v_1, v_2}(b, R) \lesssim \Norm{b}{\dot C^{0,\beta_1}_{x_1}(\BMO_{x_2})}$
    proving that $\|b\|_{\BMO^{\beta_1, 0}_{v_1, v_2}} \lesssim \Norm{b}{\dot C^{0,\beta_1}_{x_1}(\BMO_{x_2})}.$

    For the other direction fix $x_1, y_1\in \R^{d_1}$ with $x_1\neq y_1$.
    By the now very familiar techniques we see that that for all cubes $I_2 \subset \R^{d_2}$ there holds that 
    \[
        \fint_{I_2} |b_{x_1, y_1}-\langle b_{x_1, y_1} \rangle_{I_2}| \lesssim
        \|b\|_{\BMO^{\beta_1, 0}_{1, 1}} |x_1-y_1|^{\beta_1} 
        \le \|b\|_{\BMO^{\beta_1, 0}_{v_1, v_2}} |x_1-y_1|^{\beta_1}. 
    \]
    The estimate $\Norm{b}{\dot C^{0,\beta_1}_{x_1}(\BMO_{x_2})} \lesssim \|b\|_{\BMO^{\beta_1, 0}_{v_1, v_2}}$
    follows.
\end{proof}
Again, there is no dependence of $v_1, v_2$ and we can just use
$\BMO^{\beta_1, 0}$. We then move on to the corresponding
$\VMO$ space $\VMO^{\beta_1, 0} := \VMO^{\beta_1, 0}_{1,1}$.
We ask what is the right way to characterize $\VMO^{\beta_1, 0}$ in some $\VC$ type language -- similarly
as we did above for the space $\VMO^{\beta_1, \beta_2}$ with both $\beta_1, \beta_2 > 0$.
One way that will be useful to us later, in connection with approximation,
is to think about this in terms of uniform
inclusions like in Lemma \ref{lem:biVCuniform}. So rather than trying to define some space explicitly,
we go directly for this kind of uniform one-parameter inclusions.
\begin{defn}
    We write $
	G\subset_u \VMO(\R^n)$ provided that 
	for all $\varepsilon>0,$ there exists $t>0$ such that if $\ell(J)<t$ or $\ell(J)>t^{-1}$ or $\dist(J,0)>t^{-1}$ then 
	\[
	\sup_{g\in G}\fint_J\left| g - \ave{g}_J  \right|< \varepsilon.
	\]
\end{defn}
\begin{rem}\label{rem:VMOu->BMOu} It is easy to show that
	$$
	G\subset_u \VMO(\R^n)\Longrightarrow \sup_{g\in G}\| g\|_{\BMO(\R^n)}\lesssim 1,
	$$
	and hence this does not need to be included in the definition.
\end{rem}

\begin{lem}\label{lem:VMOa0uniform} Let $\beta_1>0$ and $\beta_2\in\R.$
    Then $f\in \VMO^{(\beta_1,\beta_2)}(\R^{d})$ if and only if 
	\begin{align}\label{eq:lem:VMOa0uniform}
		f\in\VC^{0,\beta_1}(\R^{d_1},\BMO^{\beta_2}(\R^{d_2})),
        \qquad \left\{ f_{x_1,y_1}^{\beta_1} \right\}_{\substack{x_1,y_1 \in\R^{d_1} \\ x_1\not= y_1}}
        \subset_u\VMO^{\beta_2}(\R^{d_2}).
	\end{align}
\end{lem}
\begin{rem}
    \begin{itemize}
       \item 
           We mostly care about the case $\beta_2 = 0$ of this lemma, and that is the case
           directly related to the current discussion.
           However, the value of $\beta_2$ does not really matter in the proof.
       \item For $\beta_2 > 0$ this is a characterization of the purely fractional
           rectangular $\VMO$. But so is Lemma \ref{lem:VMO_VC_pure_frac}.
           In fact, it appeas that Lemma \ref{lem:VMO_VC_pure_frac}
           could be derived from Lemma \ref{lem:VMOa0uniform} (going through
           Lemma \ref{lem:biVCuniform}) if we would check the one-parameter
           result that $G \subset_u \VMO^{\alpha}$ for $\alpha > 0$ if and only if
           $G \subset_u \VC^{0, \alpha}$. 
        \item Many of the techniques to prove Lemma \ref{lem:VMOa0uniform}
            have already appeared in the proofs above -- we give the full details,
            regardless.
    \end{itemize}
\end{rem}

\begin{proof}[Proof of Lemma \ref{lem:VMOa0uniform}]
	After reading the proof below in the special case $\beta_2 = 0$ it becomes
	clear that the general case $\beta_2 \in \R$ follows simply by carrying
	the extra factor $\ell(I_2)^{-\beta_2}$ around. So we assume $\beta_2 = 0$
    and relabel $\alpha = \beta_1$.
	
	Notice that the conditions on the line \eqref{eq:lem:VMOa0uniform}
	can be stated as the following two bullets:
	for all $\varepsilon>0,$ there exists $t>0$ such that whenever either 
	\begin{itemize}
		\item $|x_1-y_1|<t$ or $|x_1|>t^{-1}$ or $|y_1|>t^{-1}$ and 
		$I_2 \subset\R^{d_2}$ is an arbitrary cube, or
		\item  $\ell(I_2)<t$ or $\ell(I_2)>t^{-1}$ or $\dist(I_2,0)>t^{-1},$
		and $x_1\not=y_1\in\R^{d_1}$ are arbitrary points
	\end{itemize}  
	we have
	\begin{align}\label{eq:happo}
		\fint_{I_2} \left| f_{x_1,y_1}^{\alpha}(x_2) 
		- \ave{f_{x_1,y_1}^{\alpha}}_{J}\right|\ud x_2 \leq \varepsilon.
	\end{align}
	To see that the above two bullets imply $f\in \VMO^{(\alpha,0)}(\R^{d}),$
	fix a rectangle $R = I_1 \times I_2$ and, to control
	$\calO^{\alpha, 0}(f, R)$,
	note that 
	\begin{align*}
		&\ell(I_1)^{-\alpha} \fint_{I_1}  \fint_{I_2}\left|  f-\langle f\rangle_{I_1} 
		-\langle f\rangle_{I_2} +\langle f\rangle_{R} \right| \\ 
		&\lesssim   \fint_{I_2} \ell(I_1)^{-\alpha} \fint_{I_1} \fint_{I_1} \left|  f(x_1,x_2)
		- \langle f(x_1,\cdot) \rangle_{I_2} - \left( f(y_1,x_2)
		- \langle f(y_1,\cdot) \rangle_{I_2} \right)\right| \ud y_1\ud x_1\ud x_2 \\
		&\lesssim_{\alpha}   \fint_{I_2} \fint_{I_1} \fint_{I_1} \left| f_{x_1,y_1}^{\alpha}(x_2)
		- \ave{f_{x_1,y_1}^{\alpha}}_{I_2}\right|\ud y_1\ud x_1\ud x_2 \\
		&=  \fint_{I_1} \fint_{I_1} \left[ \fint_{I_2}  \left| f_{x_1,y_1}^{\alpha}(x_2)
		- \ave{f_{x_1,y_1}^{\alpha}}_{I_2}\right| \ud x_2\right] \ud y_1\ud x_1.
	\end{align*}
	Now, this is directly small by \eqref{eq:happo} in all the cases demanded by the definition
	of $\VMO$ except possibly when all we know that $\ell(I_1)$ is big. We have seen an argument like
	this before though -- so fix $t_0 < t$ and consider $\ell(I_1) > t_0^{-1}$.
	Then fix $x_1 \in I_1$ and notice that it is enought to bound
	\begin{align*}
		& \frac{1}{|I_1|} \int_{I_1 \cap \{|y_1| \le t^{-1}\}} \left[ \fint_{I_2}  \left| f_{x_1,y_1}^{\alpha}(x_2) 
		- \ave{f_{x_1,y_1}^{\alpha}}_{I_2}\right| \ud x_2\right] \ud y_1 \\
		&\le \frac{1}{|I_1|} \int_{I_1 \cap \{|y_1| \le t^{-1}\}} \| 
		f_{x_1,y_1}^{\alpha}\|_{\BMO(\R^{d_2})} \ud y_1 \lesssim \frac{|B(0, t^{-1})|}{|I_1|}
		\lesssim t^{-d_1}t_0^{d_1},
	\end{align*}
	where we used Remark \ref{rem:VMOu->BMOu}. This is small with a small enough $t_0$ so we are done.
	
	For the other direction,
	fix $\varepsilon > 0$ and choose $t > 0$ so that
	$$
	\calO^{\alpha, 0}(b, R) < \varepsilon, \qquad R = I_1 \times I_2,
	$$
	whenever $\ell(I_i) < t$ or $\ell(I_i) > t^{-1}$ for $i=1$ or $i=2$, or
	$\dist(R, \{0\}) > t^{-1}$. Fix $x_1 \ne y_1$ and $I_2$.
	We write
	\begin{align*}
		f_{x_1,y_1}^{\alpha}(x_2) - \ave{f_{x_1,y_1}^{\alpha}}_{I_2}
		= 	\frac{g_{x_2,I_2}(x_1) - g_{x_2,I_2}(y_1) }{|x_1-y_1|^{\alpha}},
	\end{align*}
	where
	\begin{align*}
		g_{x_2,I_2}(z_1) := f(z_1,x_2) - \ave{f(z_1,\cdot)}_{I_2}.
	\end{align*}
	With the same notation as around the line \eqref{eq:expand} we expand 
	\begin{align*}
		&g_{x_2,I_2}(x_1) - g_{x_2,I_2}(y_1) \\ 
		&= \sum_{k=0}^{\infty} [\ave{g_{x_2,I_2}}_{Q_ {k+1}(x_1)} -
		\ave{g_{x_2,I_2}}_{Q_k(x_1)}] -  
		\sum_{k=0}^{\infty} [\ave{g_{x_2,I_2}}_{Q_ {k+1}(y_1)} -  \ave{g_{x_2,I_2}}_{Q_k(y_1)}] \\
		&= I(x_2) - II(x_2).
	\end{align*}
	Then, we have 
	\begin{align*}
		\fint_{I_2}\left|I(x_2)\right|\ud x_2 &\leq \sum_{k=0}^{\infty} 
		\fint_{I_2} \left| \ave{g_{x_2,I_2}}_{Q_ {k+1}(x_1)} -  \ave{g_{x_2,I_2}}_{Q_k(x_1)}\right| \ud x_2 \\
		&\lesssim  \sum_{k=0}^{\infty}   
		\fint_{I_2} \fint_{Q_k(x_1)}\left|  g_{x_2,I_2}(z_1)-  \ave{g_{x_2,I_2}}_{Q_k(x_1)}\right| \ud z_1 \ud x_2\\
		&= \sum_{k=0}^{\infty}    \fint_{Q_k(x_1)} \fint_{I_2} |f-\langle f \rangle_{Q_k(x_1)} 
		- \langle f \rangle_{I_2} + \langle f \rangle_{Q_k(x_1) \times I_2}| \\
		&\lesssim \sum_{k=0}^{\infty} (2^{-k} |x_1-y_1|)^{\alpha} \calO^{\alpha,0}(f, Q_k(x_1)\times I_2).
	\end{align*}
	Provided that $\ell(I_2)<t$ or $\ell(I_2)>t^{-1}$ or $\dist(I_2,0) > t^{-1}$,
	we get
	\begin{align}\label{eq:perkele}
		\sum_{k=0}^{\infty} 2^{-\alpha k} \calO^{\alpha,0}(f, Q_k(x_1)\times I_2)
		\le \varepsilon \sum_{k=0}^{\infty} 2^{-\alpha k} \lesssim \varepsilon.
	\end{align}
	Of course, the symmetric bound is valid for the term $II(x_2)$.
	Thus, we obtain 
	\begin{align*}
		\fint_{I_2}\left| f_{x_1,y_1}^{\alpha}(x_2) - \ave{f_{x_1,y_1}^{\alpha}}_{I_2}\right|\ud x_2 
		&=  \fint_{I_2}\left|  \frac{g_{x_2,I_2}(x_1) - g_{x_2,I_2}(y_1) }{|x_1-y_1|^{\alpha}}  \right|\ud x_2 
		\lesssim \varepsilon
	\end{align*}
	under these assumptions on $I_2$ -- this shows that the second bullet holds.

	Then, let us look at the cases, where we have control on the variables $x_1, y_1$.
	We continue from \eqref{eq:perkele} with the goal to again show that
	$$
	\sum_{k=0}^{\infty} 2^{-\alpha k} \calO^{\alpha,0}(f, Q_k(x_1)\times I_2)
	\lesssim \varepsilon
	$$
	under such conditions on $x_1, y_1$ (independently of $I_2$). First, if $|x_1-y_1| < t/2$,
	then we have $\ell(Q_{k}(x_1)) = 2 \cdot 2^{-k}|x_1-y_1| \le
	2|x_1-y_1| < t$, and the desired bound follows as $\calO^{\alpha, 0}(f, Q_{k}(x_1) \times I_2) \le
	\varepsilon$.
	
	Choose some $t_0 < t$.
	Suppose then $|x_1| \ge Ct_0^{-1}$. We consider the cases $|x_1-y_1| < t_0^{-1}$
	and $|x_1-y_1| \ge t_0^{-1}$ separately. Suppose first $|x_1-y_1| < t_0^{-1}$.
	In this case
	\begin{align*} 
		\dist(Q_{k}(x_1) \times I_2, \{0\}) 
		\ge \dist\Big( Q\Big(\frac12(x_1+y_1), \frac12|x_1-y_1| \Big), \{0\}\Big),
	\end{align*}
	where 
	$$
	|(x_1+y_1)/2| \ge |x_1| - |x_1-y_1|/2 \ge \frac{C}{2}t_0^{-1}.
	$$
	So, if $C$ is chosen suitably, 
	$$
	\dist(Q_k(x_1) \times I_2, \{0\}) > t_0^{-1} > t^{-1}
	$$
	for all $k$. So in this case again $\calO^{\alpha,0}(f, Q_k(x_1)\times I_2) < \varepsilon$
	all the time so we are good.
	
	We then consider the other subcase $|x_1-y_1| \ge t_0^{-1}$ and split
	\begin{align*}
		\sum_{k=0}^{\infty} 2^{-\alpha k} \calO^{\alpha,0}(f, Q_k(x_1)\times I_2) &= A_1 + A_2, \\
		A_1&:= \sum_{k:\, \ell(Q_{k}(x_1)) > t^{-1}}  2^{-\alpha k} \calO^{\alpha,0}(f, Q_k(x_1)\times I_2), \\
		A_2&:= \sum_{k:\, \ell(Q_{k}(x_1)) \le t^{-1}}  2^{-\alpha k} \calO^{\alpha,0}(f, Q_k(x_1)\times I_2). \\
	\end{align*}
	It is obvious that $A_1 \lesssim \varepsilon$.
	For $A_2$ notice that
	$$
	2 \cdot 2^{-k}t_0^{-1} \le 2\cdot2^{-k}|x_1-y_1| = \ell(Q_{k}(x_1)) \le t^{-1}
	$$
	implies $2^{-k} < t^{-1}t_0.$ For $A_2$ we only use the above and 
	$\calO^{\alpha,0}(f, Q_k(x_1)\times I_2) \lesssim 1$.
	It follows that 
	\begin{align*}
		A_2 &\lesssim
		\sum_{k:\, 2^{-k} <  t^{-1}t_0} 2^{-\alpha k} \sim (t^{-1}t_0)^{\alpha} \le \varepsilon
	\end{align*}
	provided that $t_0 =  t_0(\varepsilon)$ is small enough. The case where $|y_1|$ is large
	instead is symmetric. Thus, we are done.
\end{proof}
We will complement our knowledge of the various spaces in the next section
by proving approximation results (density of smooth functions).

\section{Approximation in the relevant bi-parameter spaces}
We will prove the compactness of bi-commutators directly in the purely fractional case
by using Lemma \ref{lem:VMO_VC_pure_frac}. From this perspective, it would not be
strictly necessary to prove Theorem \ref{thm:approx:VCVC}, the corresponding approximation result
in the purely fractional case (we will only need the inclusion of $C^{\infty}_c$ in the purely
fractional $\VMO$, not the density). However,
we think the result is of independent interest and gives an additional characterization of
the purely fractional $\VMO$ space. In addition,
the techniques used in this proof are important in the other approximation
results later. Importantly, in other ranges of exponents the approximability
of the symbols is critical to us, since
to be able to derive the compactness in those ranges from the purely fractional case,
we need to go through approximation. In the diagonal, we have chosen to use the approximability
as the very definition of the product $\VMO$. However, in other cases the definitions are not
via approximation, and so will explicitly need Theorem \ref{thm:approx:VCVMO}.
\subsection{Approximation in $\dot{C}^{0,\alpha}$-valued H\"older spaces} 
\begin{thm}\label{thm:approx:VCVC} Let $\R^d=\R^{d_1}\times\R^{d_2}$ and $\vec{\beta} = (\beta_1,\beta_2)$ be such that $\beta_i\in (0,1).$ Then, there holds that 
	\[
	\VC^{0,\vec{\beta}}(\R^d) = \overline{C^{\infty}_c(\R^d)}^{\dot{\op{C}}^{0,\vec{\beta}}(\R^d)}.
	\]
\end{thm}

To state the main technical lemma of this section, let us recall the following function from Mudarra and Oikari \cite{MudOik24},
\begin{align*}
	\tau_M^i:\R^{d_i} \to B(0,M),\qquad		\tau_M^i(x) = \begin{cases}
		x,\qquad &|x|<M, \\
		\left(\frac{2M-|x|}{M}\right)^{2} x,\qquad M\leq & |x|<2M,\\
		0,\qquad &|x| \geq 2M.
	\end{cases}
\end{align*}
It is easy to see that $\tau_M^i$ is Lipschitz, in particular $\op{Lip}(\tau_M^i)\leq 5$
by the proof of \cite[Theorem 1.8.]{MudOik24}.
We also take a bump $\varphi_1^i \in C^{\infty}_c(\R^{d_i})$
such that $ \int_{\R^{d_i}} \varphi_1^i = 1$ and then form a standard approximation to identity
\begin{align*}
	\varphi_{1/M}^i(z_i) := M^{d_i}\varphi_1^i(Mz_i).
\end{align*}
The following Lemma \ref{lem:MudOik} is \cite[Corollary 2.5.]{MudOik24}, where all the
properties from the one parameter setting have been gathered that will be used in the current bi-parameter extension.
\begin{lem}[\cite{MudOik24}]\label{lem:MudOik} Let $\alpha\in (0,1)$ and $X$ be a Banach space.
	Let $G\subset_u 	\VC^{0,\alpha}(\R^n,X),$ then
	\begin{itemize}
		\item for all $L>0$ we have $G\circ\tau_L\subset_u 	\VC^{0,\alpha}(\R^n,X),$
		\item for all $L,L'>0$ we have $(G\circ\tau_L)*\varphi_{L'}\subset_u 	\VC^{0,\alpha}(\R^n,X),$
		\item for all $M=M(\varepsilon)>0$ sufficiently large 
            \begin{align*}
			\sup_{f\in G} \|  f - f\circ\tau_M  \|_{\dot{ \op{C}}^{0,\alpha}(\R^n,X)}  \leq \varepsilon,
		\end{align*}
		\item  for all $K = K(M,\varepsilon)>0$ sufficiently large
		\begin{align*}
			\sup_{f\in G}\|  f\circ\tau_{M} -( f\circ\tau_M)*\varphi_{1/K}  
			\|_{\dot{ \op{C}}^{0,\alpha}(\R^n,X)} \leq \varepsilon.
		\end{align*}
	\end{itemize}
\end{lem}

\begin{proof}[Proof of Theorem \ref{thm:approx:VCVC}]
	We begin by showing $C^{\infty}_c(\R^d) \subset \VC^{0,\vec{\beta}}(\R^d)$. 
	So let $f\in C^{\infty}_c(\R^d)$ and denote (not to be confused with the Laplace)
	\begin{equation}\label{eq:deltaop}
		\begin{split}
			\Delta f(x,y) &= f(x_1,x_2)-f(x_1,y_2)-f(y_1,x_2)+f(y_1,y_2) \\
			&= f_{x_1, y_1}(x_2) - f_{x_1, y_1}(y_2)
			= f_{x_2, y_2}(x_1) - f_{x_2, y_2}(y_1).
		\end{split}
	\end{equation} 
	We have
	\begin{equation*}
		\begin{split}
			|\Delta f(x,y)| \le \|\nabla_{\R^{d^2}} f_{x_1, y_1}\|_{L^{\infty}} |x_2-y_2|,
		\end{split}
	\end{equation*}
	where
	$$
	\|\nabla_{\R^{d^2}} f_{x_1, y_1}\|_{L^{\infty}} 
	\lesssim \min(1, \sup_{j,k} \|\partial_j \partial_k f\|_{L^{\infty}}|x_1-y_1|)
	\lesssim \min(1, |x_1-y_1|).
	$$
	It follows that
	\begin{align*}
		\babs{\Delta f(x,y)} \lesssim
		\min\big(1,\,  \abs{x_2-y_2}\abs{x_1-y_1},  \abs{x_1-y_1},  \abs{x_2-y_2}\big).
	\end{align*}
	Let $\varepsilon > 0$. Suppose $|x_1-y_1| < t$, where $t = t(\varepsilon)$ is small.
	Notice that $|x_1-y_1| \le t^{1-\beta_1} |x_1-y_1|^{\beta_1}$.
	Then we get
	$$
	|\Delta f(x,y)| \lesssim |x_1-y_1|\min(1, |x_2-y_2|)
	\le t^{1-\beta_1} |x_1-y_1|^{\beta_1} |x_2-y_2|^{\beta_2}.
	$$
	This shows that $|\Delta f(x,y)| \le \varepsilon |x_1-y_1|^{\beta_1} |x_2-y_2|^{\beta_2}$
	if $|x_1-y_1| < t$. By symmetry, this is true also if $|x_2-y_2| < t$.
	
	We need to derive the above estimate for $|\Delta f(x,y)|$ also if
	$|x_i| > t^{-1}$ or $|y_i| > t^{-1}$ for either $i=1$ or $i=2$.
	Suppose for instance $|y_1| > t^{-1}$ -- the other cases are symmetric.
	We assume $t$ is so small
	that this implies $f(y_1, z_2) = 0$ for all $z_2$ using that $f$
	is compactly supported. We may then assume $|x_1| \lesssim 1$, since
	otherwise $\Delta f(x,y) = 0$. This yields $|x_1-y_1| \gtrsim |y_1| > t^{-1}$.
	We first estimate
	\begin{align*}
		|\Delta f(x,y)| \lesssim \min(1, |x_2-y_2|) \le |x_2-y_2|^{\beta_2}.
	\end{align*}
	Then we simply write
	$$
	1 = |x_1-y_1|^{-\beta_1} |x_1-y_1|^{\beta_1},
	$$
	where, using $|x_1-y_1| \gtrsim t^{-1}$, we get
	$$
	|x_1-y_1|^{-\beta_1} \lesssim t^{\beta_1},
	$$
	and so
	$$  
	|\Delta f(x,y)| \lesssim t^{\beta_1} |x_1-y_1|^{\beta_1} |x_2-y_2|^{\beta_2}
	$$
	implying the desired estimate. We have
	proved $C^{\infty}_c \subset \VC^{0,\vec{\beta}}$. The inclusion
	$\overline{C^{\infty}_c}^{\dot{\op{C}}^{0,\vec{\beta}}} \subset
	\VC^{0,\vec{\beta}}$ follows. Indeed,
	if $f_i \to f$ in $\dot C^{0,\vec{\beta}}$, where $f_i \in \VC^{0,\vec{\beta}}$,
	it follows very easily that $f \in \VC^{0,\vec{\beta}}$. First, notice that
	\begin{align*}
		|\Delta f(x,y)| &\le |\Delta(f-f_i)(x,y)| + |\Delta f_i(x,y)| \\
		&\le \|f-f_i\|_{\dot C^{0,\vec{\beta}}}|x_1-y_1|^{\beta_1}|x_2-y_2|^{\beta_2}
		+ |\Delta f_i(x,y)|.
	\end{align*}
	It remains to fix $\varepsilon > 0$ and then $i$ so that $\|f-f_i\|_{\dot C^{0,\vec{\beta}}} < \varepsilon$,
	and then find the desired $t$ by using that $f_i \in \VC^{0,\vec{\beta}}$.
	
	For the other inclusion fix $f\in \VC^{0,\vec{\beta}}(\R^{d})$ -- we will construct a
	smooth compactly supported approximation. First, define
	\begin{align}
		g_M^2(x)  := \left(f(x_1,\cdot)\circ \tau_M^2\right)(x_2).
	\end{align}
	In what follows we will use the above obvious interpretation when taking, for instance,
	composition or convolutions with function defined only in one of the parameters.
	
	Let $\varepsilon > 0$. We claim that if $M = M(\varepsilon)>0$ is sufficiently large, then 
	\begin{align}\label{eq:X1}
		g_M^2 \in \VC^{0,\vec{\beta}}(\R^{d}),\qquad 
		\| f-g_M^2 \|_{\dot{\op{C}}^{0,\vec{\beta}}(\R^d)}\leq \varepsilon.
	\end{align}
	By Lemma \ref{lem:biVCuniform} and Lemma \ref{lem:MudOik} we choose $M$ so large that 
	\[
	\sup_{\substack{x_1,y_1 \in\R^{d_1} \\ x_1\not= y_1}}   
	\| f_{x_1,y_1}^{\beta_1} - f_{x_1,y_1}^{\beta_1}\circ \tau_M^2\|_{\dot{\op{C}}^{0,\beta_2}(\R^{d_2})} 
	\leq \varepsilon
	\]
	and this gives us the bound 
	\begin{align*}
		&\frac{|\Delta (f-f\circ\tau_M^2)(x,y)|}{|x_1-y_1|^{\beta_1}}
		= \frac{|\Delta f(x,y) - \Delta\left( f\circ\tau_M^2\right)(x,y)|}{|x_1-y_1|^{\beta_1}} \\ 
		&=  \Big| \left(  f_{x_1,y_1}^{\beta_1}(x_2) 
		-  \left(f_{x_1,y_1}^{\beta_1}\circ\tau_M^2\right)(x_2)\right)
		-\left(   f_{x_1,y_1}^{\beta_1}(y_2) - \left(f_{x_1,y_1}^{\beta_1}\circ\tau_M^2\right)(y_2)\right) \Big| \\
		&\leq \| f_{x_1,y_1}^{\beta_1} - f_{x_1,y_1}^{\beta_1}\circ \tau_M^2\|_{\dot{\op{C}}^{0,\beta_2}(\R^{d_2})} 
		|x_2-y_2|^{\beta_2} \leq \varepsilon |x_2-y_2|^{\beta_2}.
	\end{align*}
	This shows $\| f-g_M^2 \|_{\dot{\op{C}}^{0,\vec{\beta}}(\R^d)}\leq \varepsilon$.
	To check the claim on the left hand side of \eqref{eq:X1} notice that
	by Lemma \ref{lem:biVCuniform} we need to check that
	\[
	\big\{ (g_M^2)_{x_i,y_i}^{\beta_i} \big\}_{\substack{x_i,y_i \in\R^{d_i} 
			\\ x_i\not= y_i}}\subset_u \VC^{0,\beta_j}(\R^{d_j}).
	\] 
	The case  $i=1$ and $j=2$ is exactly the first bullet from Lemma \ref{lem:MudOik}.
	For the case $i=2$ and $j=1,$ using that $\Lip(\tau_M^2)\leq 5$ we have 
	\begin{align*}
		&\frac{|(f\circ\tau_M^2)_{x_2,y_2}^{\beta_2}(x_1)
			-(f\circ\tau_M^2)_{x_2,y_2}^{\beta_2}(y_1)|}{|x_1-y_1|^{\beta_1}}  
		= \frac{\big| \Delta f(x_1, \tau_M^2(x_2), y_1, \tau_M^2(y_2)) \big|}
		{|x_1-y_1|^{\beta_1}|x_2-y_2|^{\beta_2}} \\ 
		&\qquad \lesssim \frac{\Big| \Delta f(x_1, \tau_M^2(x_2), y_1, \tau_M^2(y_2))\Big|}
		{|x_1-y_1|^{\beta_1}|\tau_M^2(x_2)-\tau_M^2(y_2)|^{\beta_2}}
		= \frac{  \big|  f_{\tau_M^2(x_2),\tau_M^2(y_2)}^{\beta_2}(x_1)
			-  f_{\tau_M^2(x_2),\tau_M^2(y_2)}^{\beta_2}(y_1) \big|}{|x_1-y_1|^{\beta_1}}
	\end{align*}
	and now the claim follows immediately from the assumption $f\in \VC^{0,\vec{\beta}}(\R^d),$ or from the line \eqref{eq:biVCequiv} of Lemma \ref{lem:biVCuniform}.
	
	Now, since $g_M^2$ satisfies the same assumptions as $f$ did, repeating the above argument shows that for all $M' = M'(\varepsilon)>0$ sufficiently large,
	\begin{align}\label{eq:X2}
		g_{M',M}^{1,2} := g_M^2 \circ\tau_{M'}^1 \in \VC^{0,\vec{\beta}}(\R^{d}),\qquad \| g_M^2-  	g_{M',M}^{1,2}\|_{\dot{\op{C}}^{0,\vec{\beta}}(\R^d)}\leq \varepsilon.
	\end{align}
	Clearly, the function $g_{M',M}^{1,2}$ is boundedly supported on $\R^d.$
	
	Next, we arrange smoothness.
	We claim that if $K = K(\varepsilon)>0$ is sufficiently large, then 
	\begin{align}\label{eq:X3}
		h_K^2:= g_{M',M}^{1,2} *\varphi_{1/K}^2\in\VC^{0,\vec{\beta}}(\R^d),\qquad \|  h_K^2 - g_{M',M}^{1,2}  \|_{\dot C^{0,\vec{\beta}}(\R^d)}\leq \varepsilon.
	\end{align}
	From $g_{M',M}^{1,2} \in \VC^{0,\vec{\beta}}(\R^d)$ the  second and fourth bullets of Lemma \ref{lem:MudOik} show that if $K = K(\varepsilon)>0$ is sufficiently large, then 
	\begin{align}\label{eq:aye1}
		\Big\{ \left(h_K^2\right)_{x_1,y_1}^{\beta_1} \big\}_{\substack{x_1,y_1 \in\R^{d_1} \\ x_1\not= y_1}}\subset_u \VC^{0,\beta_2}(\R^{d_2})
	\end{align}
	and
	\begin{equation}\label{eq:aye2}
		\begin{split}
			&\sup_{\substack{x_1,y_1 \in\R^{d_1} \\ x_1\not= y_1}} \Big\| \left(h_K^2 - g_{M',M}^{1,2}\right)_{x_1,y_1}^{\beta_1}\Big\|_{\dot C^{0,\beta_2}(\R^{d_2})}  \\
			&=\sup_{\substack{x_1,y_1 \in\R^{d_1} \\ x_1\not= y_1}} \Big\| (h_K^2)_{x_1,y_1}^{\beta_1} - (h_K^2)_{x_1,y_1}^{\beta_1}*\varphi_{1/K}^2\Big\|_{\dot C^{0,\beta_2}(\R^{d_2})}  \leq  \varepsilon.
		\end{split}
	\end{equation}
	For the claim on the right hand side of \eqref{eq:X3} we use \eqref{eq:aye2} together with
	$$
	\sup_{x_1\not=y_1\in\R^{d_1}}\| \psi_{x_1,y_1}\|_{\dot C^{0,\beta_2}(\R^{d_2})}
	= \| \psi\|_{\dot{\op{C}}^{0,\vec{\beta}}(\R^d)}.
	$$
	To check the other claim on the line \eqref{eq:X3}, since we already know  \eqref{eq:aye1},
	by Lemma \ref{lem:biVCuniform} it remains to verify that 
	\begin{align*}
		\Big\{ \left(h_K^2\right)_{x_2,y_2}^{\beta_2} \Big\}_{\substack{x_2,y_2 \in\R^{d_2}
				\\ x_2\not= y_2}}\subset_u \VC^{0,\beta_1}(\R^{d_1}).
	\end{align*}
	This follows from $g_{M',M}^{1,2} \in \VC^{0,\vec{\beta}}(\R^{d})$,
	Lemma \ref{lem:biVCuniform} and the bound 
	\begin{align*}
		\big|  \left(h_K^2\right)_{x_2,y_2}^{\beta_2}(x_1) 
		&- \left(h_K^2\right)_{x_2,y_2}^{\beta_2}(y_1) \big|  
		= \left|  \int_{\R^{d_2}}\varphi_{1/K}^2(h)\frac{\Delta 
			g_{M',M}^{1,2}\left((x_1,x_2-h),(y_1,y_2-h)\right)}{|x_2-y_2|^{\beta_2}}\ud h  \right| \\                                                    
		&\leq \sup_{h\in\R^{d_2 }}\frac{\left|\Delta g_{M',M}^{1,2}
			\left((x_1,x_2-h),(y_1,y_2-h)\right) \right|}{|(x_2-h)-(y_2-h)|^{\beta_2}} \\ 
		&= \sup_{h\in\R^{d_2 }} \left|  \big(g_{M',M}^{1,2}\big)_{x_2-h,y_2-h}^{\beta_2}(x_1)
		-  \big(g_{M',M}^{1,2}\big)_{x_2-h,y_2-h}^{\beta_2}(y_1)  \right|.
	\end{align*}
	So we have now checked both claims on the line \eqref{eq:X3}, and in particular this
	allows us to repeat the above argument once again to obtain the approximation
	\begin{align}\label{eq:X4}
		h_{K',K}^{1,2}:= h_K^2 *\varphi_{1/K'}^2\in\VC^{0,\vec{\beta}}(\R^d),
		\qquad \|  h_K^2 - h_{K',K}^{1,2}  \|_{\dot C^{0,\vec{\beta}}(\R^d)}\leq \varepsilon,
	\end{align}
	valid for $K' = K'(\varepsilon)$ large enough.
	To conclude, let $N = \max(M,M',K,K')$ and  combine  \eqref{eq:X1}, \eqref{eq:X2},
	\eqref{eq:X3}, \eqref{eq:X4} to obtain
	$$
	\widetilde{f}_N := \left( \left( \left( f\circ\tau_N^2\right) 
	\circ\tau_{N}^1\right)*\varphi_{1/N}^1 \right)*\varphi_{1/N}^2,\qquad
	\|  f - \widetilde{f}_N  \|_{\dot C^{0,\vec{\beta}}(\R^d)}\leq 4\varepsilon;
	$$ 
	clearly $\widetilde{f}_N$ is smooth and boundedly supported.
\end{proof}

\subsection{Approximation in $\BMO$-valued H\"older spaces} 
\begin{thm}\label{thm:approx:VCVMO} Let $\R^d=\R^{d_1}\times\R^{d_2}$ and $\alpha \in (0,1).$
	Then, there holds that 
	\[
	\VMO^{\alpha,0}(\R^{d})= \overline{C^{\infty}_c(\R^d)}^{\BMO^{\alpha,0}(\R^d)}.
	\]
\end{thm}

\begin{proof}[Proof of Theorem \ref{thm:approx:VCVMO}, $``\supset"$ direction]
	We begin by showing $C^{\infty}_c(\R^d)\subset 	\VMO^{\alpha,0}(\R^{d})$.
	Recall $\Delta f$ from \eqref{eq:deltaop} and the pointwise estimates for it
	for $f \in C^{\infty}_c$. Fix a rectangle $R = I_1 \times I_2$ and $\varepsilon > 0$.
	We have the standard estimate
	\begin{align*}
		&\calO^{\alpha,0}(f, R) \lesssim \ell(I_1)^{-\alpha}\fint_R \fint_R
		|\Delta f(x,y)| \ud y \ud x \\ 
		&\lesssim \ell(I_1)^{-\alpha} \fint_{I_1} \fint_{I_2} \fint_{I_1} \fint_{I_2}
		\min\left(1,|x_1-y_1|, |x_2-y_2|, |x_1-y_1||x_2-y_2|\right) \ud y_2\ud y_1\ud x_2\ud x_1 \\
		&\lesssim  \ell(I_1)^{-\alpha}\min\left( 1, \ell(I_1),\, \ell(I_2),\, \ell(I_1)\ell(I_2) \right). 
	\end{align*}
	Let $t = t(\varepsilon) > 0$ be small enough. Suppose $\ell(I_1) < t$ -- then
	$$
	\calO^{\alpha,0}(f, R) \lesssim \ell(I_1)^{1-\alpha} < t^{1-\alpha}
	$$
	is small as $\alpha < 1$. If $\ell(I_2) < t$, we estimate, using $\alpha > 0$, that
	$$
	\calO^{\alpha,0}(f, R) \lesssim \ell(I_1)^{-\alpha} \min(1,\ell(I_1))\ell(I_2) \le \ell(I_2) < t.
	$$
	If $\ell(I_1) > t^{-1}$ we get
	$$
	\calO^{\alpha,0}(f, R) \lesssim \ell(I_1)^{-\alpha} < t^{\alpha}.
	$$
	It is also clear that $\Delta f = 0$ in the above integral if $\dist(R, \{0\})$
	is big enough. So the only remaining case to handle is when $\ell(I_2)$ is large.
	In this case we estimate
	\begin{align*}
		\calO^{\alpha,0}(f, R) &\lesssim \ell(I_1)^{-\alpha}\min(1, \ell(I_1)) 
		\fint_R \fint_R 1_{\supp \Delta f}(x,y) \ud y \ud x \\
		&\le \fint_{I_2} \fint_{I_2} 
		1_{\{|x_2| < C \textup{ or } |y_2| < C\}}(x_2, y_2) \ud y_2 \ud x_2
		\lesssim \frac{|B(0,C)|}{|I_2|}
	\end{align*}
	for some big $C > 0$ depending on the support of $f$. But obviously this is small
	if $|I_2|$ is large enough. Thus, $C^{\infty}_c(\R^d) \subset \VMO^{\alpha,0}(\R^{d})$
	and the claim follows as $\VMO^{\alpha, 0}$ is clearly closed in $\BMO^{\alpha, 0}$.
\end{proof}

Uchiyama  \cite[3. Lemma.]{Uch1978} constructs an approximation that can be written as a family of linear operators, parameterized by $N>0,$ of the generic form
\begin{align*}
	L_{N}^2f =  S_{N}^2(f)*\varphi_{1/N}^2,\qquad S_{N}^2:L^1_{\loc}(\R^{d_2})\to L^{\infty}(\R^{d_2}\cap B(0,r(N))),
\end{align*}
where $r(N)>0$ is a fixed radius that depends only on $N.$ In particular, $L_{N}^2$ (denoted here as such, Uchiyama does not explicitly name operators in his proof) distills out functions with uniformly bounded support.
In particular the following lemma, a quantitative version of Uchiyamas approximation, is easy to verify.
\begin{lem}[Uchiyama \cite{Uch1978}]\label{lem:ApproxUniformUchi1} For each $N>0$ the operator $L_N$ is bounded on $\BMO(\R^{d_2}).$
	Moreover, if $G\subset_u\VMO(\R^{d_2}),$ then 
	\begin{itemize}
		\item 
		for all $N>0$  there holds that 
		\[
		L_{N}^2\circ G\subset_u\VMO(\R^{d_2}),
		\]
		\item and for each $\varepsilon$ there exists $N = N(G,\varepsilon)>0$ sufficiently large so that 
		\[
		\sup_{g\in G}\| g - L_{N}^2(g) \|_{\BMO(\R^{d_2})}\leq \varepsilon.
		\]
	\end{itemize}
\end{lem}

By the linearity of $L_N^2$ we have
\[
(L_N^2f)_{x_1,y_1}^{\alpha}= \frac{L_N^2(f(x_1,\cdot )) 
	- L_N^2(f(y_1,\cdot ))}{|x_1-y_1|^{\alpha}}
=  L_N^2\left( \frac{f(x_1,\cdot ) - f(y_1,\cdot )}{|x_1-y_1|^{\alpha}} \right)
= L_N^2(f_{x_1,y_1}^{\alpha}),
\]
a fact which we use below.
As $L_N^2$ is bounded on $\BMO(\R^{d_2}),$ we obtain the following. 
\begin{lem}\label{lem:ApproxPreserveUchi} Let $N>0$ be arbitrary. Then,  
	\[
	L_N^2 \circ  \VC^{0,\alpha}(\R^{d_1},\BMO(\R^{d_2}))  \subset  \VC^{0,\alpha}(\R^{d_1},\BMO(\R^{d_2})).
	\]
\end{lem}

\begin{proof} 
	Fix $f \in \VC^{0,\alpha}(\R^{d_1},\BMO(\R^{d_2}))$. 
	Let  $\varepsilon>0$. Fix $t$
	so that if $|x_1-y_1|<t$ or $|x_1|>t^{-1}$ or $|y_1|>t^{-1},$ then
	$\| f^{\alpha}_{x_1,y_1} \|_{\BMO(\R^{d_2})}\leq \varepsilon$.
	Then, by the boundedness of $L_N^2$ on $\BMO$, if $x_1, y_1$ satisfy
	one of the above conditions, we obtain 
	\[
	\| (L_N^2f)^{\alpha}_{x_1,y_1}\|_{\BMO(\R^{d_2})}  
	= \|L_N^2(f_{x_1,y_1}^{\alpha})\|_{\BMO(\R^{d_2})}
	\lesssim \| f^{\alpha}_{x_1,y_1} \|_{\BMO(\R^{d_2})} \le \varepsilon. 
	\]
	We are done.
\end{proof}

\begin{proof}[Proof of Theorem \ref{thm:approx:VCVMO}, $``\subset"$ direction] 
	Let $f\in\VMO^{(\alpha,0)}(\R^d)$ and $\varepsilon > 0$. By Lemma \ref{lem:VMOa0uniform} we have 
	\begin{align}\label{oo}
		f\in \VC^{0,\alpha}(\R^{d_1},\BMO(\R^{d_2})),\qquad 
		\left\{ f_{x_1,y_1}^{\alpha} \right\}_{\substack{x_1,y_1 \in\R^{d_1} \\ x_1\not= y_1}}
		\subset_u\VMO(\R^{d_2}).
	\end{align}
	The uniform inclusion and Lemma \ref{lem:ApproxUniformUchi1} give for $N = N(\varepsilon)>0$ 
	large enough that
	\begin{equation}\label{eq:lit1}
		\begin{split}
			\| f - L_{N}^2(f)  \|_{\dot C^{0,\alpha}(\R^{d_1},\BMO(\R^{d_2}))} &= \sup_{\substack{x_1,y_1 \in\R^{d_1} \\ x_1\not= y_1}} \|  (f  - L_N^2(f))^{\alpha}_{x_1,y_1}  \|_{\BMO(\R^{d_2})} \\
			&= \sup_{\substack{x_1,y_1 \in\R^{d_1} \\ x_1\not= y_1}} \|  f^{\alpha}_{x_1,y_1}  - L_N^2(f^{\alpha}_{x_1,y_1}) \|_{\BMO(\R^{d_2})}\leq \varepsilon.
		\end{split}
	\end{equation}
	The inclusion on the line \eqref{oo} with Lemma \ref{lem:ApproxPreserveUchi} gives
	\begin{align}
		L_{N}^2(f) \in \VC^{0,\alpha}(\R^{d_1},\BMO(\R^{d_2})).
	\end{align}
	By Lemma \ref{lem:MudOik} with $M = M(\varepsilon)>0$ sufficiently large we have 
	\begin{align}\label{eq:lit2}
		\| L_{N}^2(f) - \left(\left(L_{N}^2(f)\right)\circ\tau_M^1\right)*\varphi_{1/M}^1 \|_{\dot{C}^{0,\alpha}(\R^{d_1}, \BMO(\R^{d_2}) )}\leq \varepsilon.
	\end{align}
	To see the bounded support and smoothness of the final approximant, write
	\begin{align}
		\left(L_{N}^2(f)\circ\tau_M^1\right)*\varphi_{1/M}^1 = 	\left(S_N^2(f)\circ\tau_M^1\right)*(\varphi_{1/M}^1\otimes \varphi_{1/N}^2) =: \wt{f}.
	\end{align}
	Moreover, having \eqref{eq:lit1} and \eqref{eq:lit2} together shows that
	\begin{align*}
		&\| f - \wt{f} \|_{\BMO^{(\alpha,0)}(\R^d)}
		\sim\| f - \wt{f} \|_{\dot C^{0,\alpha}(\R^{d_1},\BMO(\R^{d_2}))} \\ 
		&\leq 	\| f - L_{N}^2(f)  \|_{\dot C^{0,\alpha}(\R^{d_1},\BMO(\R^{d_2}))} 
		+ 	\| L_{N}^2(f) - \wt{f} \|_{\dot C^{0,\alpha}(\R^{d_1},\BMO(\R^{d_2}))} 
		\leq 2\varepsilon.
	\end{align*}
\end{proof}

\section{Sufficient conditions for compactness}
We prove various, often optimal (as discussed later in connection with the necessary conditions),
sufficient conditions for
the compactness of bi-commutators. We start by proving the strictly fractional case directly
-- many of the remaining cases are then deduced from this strictly fractional case via
an abstract principle (for related one-parameter principles see the note \cite{Oik24} by
the last named author).
\subsection{The case $p_i < q_i$, $i=1,2$}
Let $\varphi \in C^{\infty}_c(\R^d)$ with $0 \le \varphi \le 1$,
$\varphi(x) = 1$ for $|x| \le 1/2$ and $\varphi(x) = 0$ for $|x| \ge 1$.
For $x \in \R^d$ and $t > 0$ write
\begin{align*}
    L_{x, t}(y) &:= \varphi(t(y-x)), \\
    G_{x, t}(y) &:= 1 - L_{x,t}(y), \\
    1 &= L_{x,t}(y) + G_{x,t}(y).
\end{align*}
(Here $L$ is for local and $G$ for global.)
If $x = 0$ we abbreviate $L_{0,t} = L_t$, $G_{0, t} = G_t$.

Write for $0 < t < 1$ that
$$
Tf(x) = L_{t}(x)T(L_t f)(x) + L_t(x)T(G_t f)(x) + G_t(x)Tf(x).
$$
Then, we further write
$$
L_t(x)T(L_t f)(x) = L_t(x) T(L_t L_{x, t^{-1}}f)(x) + L_t(x) T(L_t G_{x, t^{-1}} f)(x).
$$
Notice that if $L_t(x) \ne 0$ and $L_t(y) G_{x, t^{-1}}(y) \ne 0$, then
$|x| < t^{-1}$, $|y| < t^{-1}$ and $|x-y| > t/2$, and so $t/2 < |x-y| < 2t^{-1}.$
Now, the decomposition is
\begin{align*}
    Tf(x) = T_{C}f(x) + T_{E}f(x),    
\end{align*}
where
$$
    T_{C}f(x) := L_t(x) T(L_t G_{x, t^{-1}} f)(x)
$$
and
$$
    T_{E}f(x) := L_t(x) T(L_t L_{x, t^{-1}}f)(x) +  L_t(x)T(G_t f)(x)
                + G_t(x)Tf(x).
$$
Here one can think that $C$ stands for compact and $E$ for error.
These operators depend on the parameter $t$ but it is not highlighted
in the notation.

Now, let $T_i$, $i = 1,2$, be a CZO on $\R^{d_i}$, and let $d = d_1 + d_2$. 
We write
$$
T_i = T_{i, C} + T_{i, E}
$$
using the above decomposition on $\R^{d_i}$ with functions $\varphi_i$,
$L_{x,t}^i$ and $G_{x,t}^i$.
The compactness of $[T_{1, C}, [b, T_{2, C}]]$ is proved using the Fr\'echet--Kolmogorov,
which we now recall. It is a general fact that Fr\'echet-Kolmogorov holds in
Banach function spaces. Indeed, that is the content of \cite[Theorem 3.1.]{GuoZha2020},
which in the special case of $L^{s,t}$
produces the following Fr\'echet-Kolmogorov for mixed Lebesgue spaces. 

\begin{lem} Let $s,t\in (1,\infty)$.
    Then, there holds that $\calF\subset L^{s,t}$
    is relatively compact if and only if the family of functions $\calF$ is  
	\begin{itemize}
		\item equibounded: $\sup_{f\in\calF}\|f\|_{L^{s,t}} \lesssim 1,$
		\item equivanishing: $\lim_{R\to\infty}\sup_{f\in\calF}\| 1_{B(0,R)^c} f\|_{L^{s,t}} = 0,$ and 
		\item equicontinuous: $\lim_{h \to 0}\sup_{f\in\calF}
            \|  f(\cdot) - f(\cdot + h)\|_{L^{s,t}} = 0.$
	\end{itemize}
\end{lem} 

\begin{lem}
    Let $1 < p_i < q_i < \infty$, $b \in C(\R^d)$ and
    $T_i$ be a CZO for $i=1,2$. Then $[T_{1, C}, [b,T_{2,C}]]$ is compact
    $L^{p_1}L^{p_2} \to L^{q_1}L^{q_2}$.
\end{lem}
\begin{proof}
    All estimates below are allowed to depend on the parameter $t$
    that is implicit in the operators $T_{i, C}$.

    Write
    \begin{equation}\label{eq:capB}
        B(x,y) = b(x_1, x_2)-b(x_1, y_2)-b(y_1,x_2)+b(y_1, y_2)
    \end{equation}
    and
    $$
        Pf(x) := [T_{1, C}, [b, T_{2, C}]]f(x) = - \int_{\R^d} B(x,y)
        K_C(x,y) f(y) \ud y, 
    $$
    where
    \begin{align*}
        K_C(x,y) &:= L(x,y)G(x,y)K(x,y), \\
        L(x,y) &:= L_t(x_1)L_t(x_2)L_t(y_1)L_t(y_2), \\
        G(x,y) &:= G_{x_1, t^{-1}}(y_1) G_{x_2, t^{-1}}(y_2), \\
        K(x,y) &:= K_1(x_1,y_1)K_2(x_2,y_2).
    \end{align*}
    Notice (using the the fact that $b \in L^{\infty}_{\loc}$) that
    $$
        |Pf(x)| \lesssim \int_{|x_1-y_1| > t/2} \int_{|x_2-y_2| > t/2}
        \frac{|f(y)|}{|x_1-y_1|^{d_1}|x_2-y_2|^{d_2}} \ud y_2 \ud y_1
        \lesssim I_{\beta_1}^1 I_{\beta_2}^2 |f|(x_1, x_2),
    $$
    where this fractional integral bound, with the explicit choice
    $\frac{\beta_i}{d_i}:= \frac{1}{p_i} - \frac{1}{q_i}$,
    is a way to see
    that $P$ maps $L^{p_1}L^{p_2} \to L^{q_1}L^{q_2}$ boundedly.

    To see compactness, fix a sequence $(f_j)_j$ with $\|f_j\|_{L^{p_1}L^{p_2}} \lesssim 1$.
    We will show that $(Pf_j)_j$ satisfies the assumptions of Fr\'echet--Kolmogorov on
    the Banach function space $L^{q_1}L^{q_2}$. The equiboundedness is trivial from the boundedness of $P$,
    while the equivanishing follows as we have for large enough $M$ (independently on $j$) that 
    $$
        \| (x_1, x_2) \mapsto 1_{B(0, M)^c}(x_1, x_2) P f_j(x_1, x_2) \|_{L^{q_1}L^{q_2}} = 0.
    $$
    Indeed, $L_t^1(x_1)L_t^2(x_2) \ne 0$ requires $|x_i| < t^{-1}$ for $i=1,2$, and so
    $1_{B(0, M)^c}(x_1, x_2) = 0$ with big $M$.

    We now study the equicontinuity of $(Pf_j)_j$. Notice that
    \begin{align*}
    |Pf_j(x+h) - Pf_j(x)|
    \le \int_{\R^d} |A(x+h, y)-A(x,y)|  |f_j(y)| \ud y,
    \end{align*}
    where $A(x,y) := B(x,y)K_C(x,y)$. Notice that $A$ is uniformly continuous on compact sets,
    and so given $\varepsilon > 0$, we have for all small enough $h$ that
    $$
        |Pf_j(x+h) - Pf_j(x)| \le \varepsilon \int_{|x_1-y_1|
        \lesssim t^{-1}} \int_{|x_2-y_2| \lesssim t^{-1}} |f_j(y)| \ud y
        \lesssim \varepsilon  I_{\beta_1}^1 I_{\beta_2}^2 |f|(x_1, x_2),
    $$
    and so
    $$
        \|Pf_j(\cdot + h) - Pf_j\|_{L^{q_1}L^{q_2}}
        \lesssim \varepsilon \|f_j\|_{L^{p_1}L^{p_2}} \lesssim \varepsilon    
    $$
    uniformly on $f_j$. This shows compactness, and we are done.
\end{proof}

We are ready to prove the compactness of bi-commutators in the strictly fractional
case under the natural and optimal condition that $b$ belongs to the corresponding
rectangular fractional $\VMO$ space (necessity is proved later).
\begin{thm}\label{thm:qq>pp}
    Suppose $1 < p_i < q_i < \infty$ and $\beta_i/d_i = 1/p_i - 1/q_i$, $i=1,2$.
    If $b \in \VMO^{\beta_1, \beta_2}$ and $T_i$ is CZO for $i=1,2$, then
    $[T_1, [b,T_2]]$ is compact $L^{p_1}L^{p_2} \to L^{q_1}L^{q_2}$.
\end{thm}
\begin{proof}
    With $t < 1$ fixed write $T_i = T_{i,C} + T_{i,E}$ as previously. We consider
    the commutators $[P_1, [b,P_2]]$, where $P_i \in \{T_{i, C}, T_{i, E}\}$
    and $P_i \ne T_{i, C}$ for at least one $i = 1,2$. Write
    $$
    [P_1, [b,P_2]]f(x) = -\int_{\R^d} B(x,y)K(x,y)f(y) \ud y,
    $$
    where $B$ is as in \eqref{eq:capB} and $K$ is some appropriate kernel
    determined by the kernels $K_1$, $K_2$ and the auxiliary functions present
    in $P_i$. If $P_i \ne T_{i, C}$ this implies that in the above integral
    $|x_i-y_i| < t$ or $|y_i| > t^{-1}/2$ or $|x_i| > t^{-1}/2$.
    This means that given $\varepsilon > 0$, using Lemma \ref{lem:VMO_VC_pure_frac},
    we have for a small enough $t$ that
    $$
        |[P_1, [b,P_2]]f(x)| \le \varepsilon I^{\beta_1} I^{\beta_2} |f|(x_1, x_2)
    $$
    and so
    $$
        \|[P_1, [b,P_2]]f \|_{L^{q_1}L^{q_2}} \lesssim \varepsilon \|f\|_{L^{p_1}L^{p_2}}.
    $$
    This shows that
    $$
        \|[T_1, [b,T_2]]-[T_{1, C}, [b,T_{2,C}]]\|_{L^{p_1}L^{p_2} \to L^{q_1}L^{q_2}} \lesssim \varepsilon
    $$
    for small $t$. Since the operators $[T_{1, C}, [b,T_{2,C}]]$ are compact, and compact operators
    form a closed subspace of bounded operators, it follows that also $[T_1, [b,T_2]]$
    is compact as an operator $L^{p_1}L^{p_2} \to L^{q_1}L^{q_2}$.
\end{proof}

\subsection{Compactness extrapolation principle for bi-commutators}
In this section we prove necessary conditions for the compactness of the bi-commutator
in many of the remaining cases. It is based on an extrapolation scheme.

First, we need some notation (familiar from the Introduction).
Given  $p_i, q_i \in (1, \infty),$ let $\beta_i,r_i$ be defined through the relations
\begin{equation*}
\begin{split}
  \beta_i := d_i\left( \frac{1}{p_i}-\frac{1}{q_i}\right), \quad \text{if}\quad p_i<q_i;\qquad
  \frac{1}{q_i} :=\frac{1}{r_i}+\frac{1}{p_i},\quad \text{if}\quad p_i>q_i.
\end{split}
\end{equation*}
Let $T_1$ and $T_2$ be two symmetrically non-degenerate CZOs on $\R^{d_1}$ and $\R^{d_2}$,
respectively, and $b \in L^2_{\loc}(\R^{d})$, $d=d_1+d_2$. 
Only the lower bounds require the CZOs to be non-degenerate -- and the lower bounds are not important
in this section, so we postpone the definition.
The main result of \cite{AHLMO2021}
tells us that $\|[T_1, [b, T_2]]\|_{L^{p_1}L^{p_2}\to L^{q_1}L^{q_2}}$
has upper and lower bounds according to the following table:
\begin{center}
    \begin{tabular}{ c | c | c | c }
        & $p_1<q_1$ & $p_1=q_1$ & $p_1>q_1$ \\ \hline
        & & & \\
        $p_2<q_2$ &  $\sim\|b\|_{\dot C^{0,\beta_1}_{x_1}(\dot C^{0,\beta_2}_{x_2})}$
                  &  $\sim\|b\|_{\dot C^{0,\beta_2}_{x_2}(\BMO_{x_1})}$	
                  & $\lesssim\|b\|_{\dot L^{r_1}_{x_1}(\dot C^{0,\beta_2}_{x_2})}$ \\
                  & & &  $\gtrsim\|b\|_{\dot C^{0,\beta_2}_{x_2}(\dot L^{r_1}_{x_1})}$ \\ \hline
                  & &  & \\
        $p_2=q_2$ & $\sim\|b\|_{\dot C^{0,\beta_1}_{x_1}(\BMO_{x_2})}$
                  & $\lesssim \|b\|_{\BMO_{\rm prod}}$
                  &  $\lesssim\|b\|_{\dot L^{r_1}_{x_1}(\BMO_{x_2})}$ \\
                  & & $\gtrsim \|b\|_{\BMO_{p_1', p_2'}}$ & \\ \hline
                  & & & \\
        $p_2>q_2$ & $\sim\| b\|_{\dot C^{0,\beta_1}_{x_1}(\dot L^{r_2}_{x_2})}$
                  & $\lesssim\| b\|_{\BMO_{x_1}(\dot L^{r_2}_{x_2})}$
                  & $\lesssim\| b\|_{\dot L^{r_1}_{x_1}(\dot L^{r_2}_{x_2})}$	\\
    \end{tabular}
\end{center}
Of course, the result is only fully satisfactory when a $\sim$ holds.
Also, we do not usually write the $x_1$ and $x_2$ subscripts but they are included in this table
for extra clarity. 

Let $\vec{p} := (p_1,p_2)$ and $\vec{q} := (q_1,q_2)$. We define the
space $X^{\vec{p},\vec{q}}$ to be the space from the above
table for which the upper bound holds:
\begin{align}\label{eq:bound:Xpq}
    \|[T_1,[b,T_2]] \|_{L^{p_1,p_2}\to L^{q_1,q_2}}
    \lesssim \| b \|_{X^{\vec{p},\vec{q}}}, \qquad 1 < p_i, q_i < \infty.
\end{align}

When $\beta_i > 1$ we only get constant functions in that parameter -- our extrapolation
method is restricted so that it only deals with $X^{\vec{p},\vec{q}}$ that are non-trivial.
\begin{thm}
    Let $p_i, q_i \in (1, \infty)$ be some fixed exponents so that $\beta_i < 1$, $i=1,2$. 
    Let $Z$ be some collection of bounded and compactly supported functions in $\R^d$ with
    $$
    Z \subset X^{\vec{p},\vec{q}} \cap \VMO^{\alpha_1, \alpha_2}
    $$
    for all $\alpha_1, \alpha_2 \in (0, 1)$. Then for CZOs $T_i$, $i = 1,2$, we have
    that $[T_1,[b,T_2]]$ is compact $L^{p_1}L^{p_2} \to L^{q_1}L^{q_2}$ 
    for all
    $$
        b \in Y^{\vec{p},\vec{q}}
        = Y^{\vec{p},\vec{q}}(\R^d)
        := \overline{Z}^{X^{\vec{p},\vec{q}}}.
    $$
\end{thm}
\begin{proof}
    Fix arbitrary exponents $1 < p_i, q_i < \infty$ with $\beta_i < 1$
    and $b\in Y^{\vec{p},\vec{q}}$. By the definition
    of $Y^{\vec{p},\vec{q}}$ we find $b_i \in Z$ so that
    $$
        \Norm{b-b_i}{X^{\vec{p},\vec{q}}}< 1/i.
    $$
    Then, by \eqref{eq:bound:Xpq} we have
    \begin{align*}
        \bNorm{[T_1,[b,T_2]] - [T_1,[b_i,T_2]]}{L^{p_1,p_2}\to L^{q_1,q_2}}
        &= \bNorm{[T_1,[b-b_i,T_2]]}{L^{p_1,p_2}\to L^{q_1,q_2}} \\
        &\lesssim \Norm{b-b_i}{X^{\vec{p},\vec{q}}} < 1/i.
    \end{align*}
    As compact operators form a closed subspace of bounded operators, it is enough to show that each 
    $[T_1,[b_i,T_2]]$ is compact $L^{p_1}L^{p_2} \to L^{q_1}L^{q_2}$.
    So without loss of generality, we may assume that $b \in Z$.

    Now, choose cubes $I \subset \R^{d_1}$ and $J \subset \R^{d_2}$ so that
    $\supp b\subset I\times J$.
    We split
    \begin{align}\label{eq:splitT1}
        T_1 = 1_IT_11_I + 1_IT_11_{I^c}+1_{I^c}T_11_I+1_{I^c}T_11_{I^c},
    \end{align}
    \begin{align}\label{eq:splitT2}
        T_2 = 1_JT_21_J + 1_JT_21_{J^c}+1_{J^c}T_21_J+1_{J^c}T_21_{J^c}.
    \end{align}
    The bi-commutator is linear in all entries, thus splitting
    it according to \eqref{eq:splitT1} and \eqref{eq:splitT2} we obtain in total
    $4\times 4 = 16$ terms. The terms that contain
    $1_{I^c}T_11_{I^c}$ or $1_{J^c}T_21_{J^c}$ vanish by the support
    condition of $b$.
    This leaves us with $3\times 3 =9$ terms that we must show to be compact.
    A representative case is the following,
    \begin{align}\label{eq:repr}
        \left[1_{I^c}T_11_I, [b , 1_JT_21_{J^c}]\right] 
        = 1_{I^c\times J}\circ \left[T_1,[b, T_2]\right]\circ1_{I\times J^c},
    \end{align}
    with the other cases being completely analogous. Here the indicator is seen as pointwise multiplication.

    We show that the operator \eqref{eq:repr} is compact by viewing
    it as a composition of two bounded and one compact mapping between suitable spaces.
    We can pick auxiliary exponents
    $p_1^-$ and $q_2^+$ so that
    \begin{align*}
        1 < p_1^- < \min(p_1, q_1) \qquad \textup{and} \qquad \max(p_2, q_2) < q_2^+ < \infty.  
    \end{align*}
    Moreover, we pick them so that
    $$
        \alpha_1 := d_1\Big(\frac{1}{p_1^-}-\frac{1}{q_1}\Big) \in (0,1),\qquad 
        \alpha_2 := d_2\Big(\frac{1}{p_2}-\frac{1}{q_2^+}\Big) \in (0,1).
    $$
    We explain why this is possible. First, regarding $\alpha_1$, consider the
    case $p_1 \ge q_1$. Then we can pick $p_1^- < q_1$ very close to $q_1$
    so that $\alpha_1 < 1$. In the case $p_1 < q_1$ with $\beta_1 < 1$ we can choose
    $p_1^- < p_1$ so close to to $p_1$ that still $\alpha_1 < 1$ (using that $\beta_1 < 1$).
    Regarding $q_2^+$ consider the case $p_2 \ge q_2$. Then pick $q_2^+ > p_2$
    so close to $p_2$ that $\alpha_2 < 1$. Finally, if $p_2 < q_2$ with $\beta_2 < 1$
    choose $q_2^+ > q_2$ so close to $q_2$ that still $\alpha_2 < 1$ (using $\beta_2 < 1$).

    Now, take a look at the following diagram,
    \begin{align}
        L^{p_1,p_2} \overset{1_{I\times J^c}}{\longrightarrow } 
        L^{p_1^-,p_2} \overset{\left[T_1,[b,T_2]\right]}{\longrightarrow }
        L^{q_1,q_2^+}  \overset{1_{I^c\times J}}{\longrightarrow } L^{q_1,q_2}.
    \end{align}
    Directly by H\"{o}lders inequality, the first and the last maps are bounded.
    The middle is compact by Theorem \ref{thm:qq>pp} as $b \in Z \subset \VMO^{\alpha_1, \alpha_2}$.
    Thus, the composition is compact.
\end{proof}

We now derive the compactness for the other exponent ranges.
\begin{itemize}
    \item Consider the diagonal $p_1 = q_1$ and $p_2 = q_2$. Then $X^{\vec p, \vec q}
        = \BMO_{\operatorname{prod}}$ and we have $Z := C^{\infty}_c \subset \BMO_{\operatorname{prod}}
        \cap \VMO^{\alpha_1, \alpha_2}$ (by the easy side of Theorem \ref{thm:approx:VCVC} and
        the discussion using litle $\bmo$ showing the same inclusion for the product $\BMO$)
        for all $\alpha_1, \alpha_2 \in (1, \infty)$.
        We can conclude that $[T_1, [b, T_2]]$ is compact, by the extrapolation theorem, for all
        $$
        b \in \overline{C^{\infty}_c}^{\BMO_{\operatorname{prod}}} = \VMO_{\operatorname{prod}}.
        $$
        Here the last equality is just by definition.
    \item Suppose $p_1 < q_1$ but $p_2 = q_2$. Now, if $\beta_1 \ge 1$, our necessity proofs
        will show that compactness requires that $b$ is such that the bi-commutator vanishes
        (a sum of two functions -- one constant in $x_1$, the other in $x_2$). So conversely,
        if $b$ is such, the bi-commutator vanishes and clearly it is compact.
        So we may assume $\beta_1 < 1$. Then $X^{\vec p, \vec q} = C^{\beta_1}(\BMO) =
        \BMO^{\beta_1, 0}$ (by Theorem \ref{thm:CBMO}) and we have $Z := C^{\infty}_c \subset 
        \BMO^{\beta_1, 0}
        \cap \VMO^{\alpha_1, \alpha_2}$ for all $0 < \alpha_1, \alpha_2 < 1$,
        and that, by the extrapolation theorem, $[T_1, [b, T_2]]$ is compact for all
        $$
            b \in \overline{C^{\infty}_c}^{\BMO^{\beta_1,0}}
            = \VMO^{\beta_1,0},
        $$
        where we used Theorem \ref{thm:approx:VCVMO} for the last equality.
    \item The case $p_1 = q_1$ and $p_2 < q_2$ is symmetric.
    \item The case $p_1 < q_1$ and $p_2 < q_2$ we already proved directly.
\end{itemize}
This ends our proof of the sufficiency claims of Theorem \ref{thm:main}.

\section{Necessary conditions for compactness}\label{sect:NecesaryConditions}
In this section we prove the corresponding necessary conditions for compactness
of bi-commutators. However,
we must start with some preliminary results and definitions.

The one-parameter version of Lemma \ref{lem:seq} is implicit in Uchiyama \cite{Uch1978}.
This one-parameter version was also later explicitly stated in
\cite[Lemma 5.21.]{HOS2023}.
\begin{lem}\label{lem:seq}
	Let $U:L^{p_1,p_2}\to L^{q_1,q_2}$ be a bounded linear operator
    with given $1<p_i,q_i<\infty.$ Then, there does \textbf{not} exist a
    sequence $\{u_k\}_{k=1}^\infty\subset L^{p_1,p_2}$ with the following properties:
	\begin{enumerate}[$(i)$]
		\item $\sup_{k\in\N}\|u_k\|_{L^{p_1,p_2}}\lesssim 1$; 
		\item for  $j=1$ \textbf{or} $j=2,$ there holds that 
		\begin{align*}
			\pi_j(\supp u_k) \cap \pi_j(\supp u_m) = \emptyset,\qquad k\not=m,
		\end{align*}
        where $\pi_j$ is the projection $\R^d \to \R^{d_j}$;
		\item there exists a non-trivial $\Phi\in L^{q_1,q_2}$ so that
            $$
                Uu_k \to \Phi
            $$
            in $L^{q_1, q_2}$.
	\end{enumerate} 
\end{lem}
\begin{proof}
    Aiming for a contradiction, suppose such a sequence $(u_k)$ exists.
	By the condition $(iii)$, passing to a subsequence if necessary, 
    we can assume $\Norm{Uu_k-\Phi}{L^{q_1,q_2}}\leq  2^{-k}$ for all $k$.
	Let $s$ be such that $s \in (1, \min(p_1, p_2))$. We fix a positive auxiliary sequence
	$
	\{a_k\}_{k=1}^{\infty}\in  \ell^s \setminus \ell^1.
	$
	Define
	\[
	g^N := \sum_{k=1}^N a_ku_k.	
	\]
    If the condition (ii) holds with $j=1$ we have, also using condition (i), that
	\begin{align*}
		\Norm{g^N}{L^{p_1,p_2}} 
        &= \Big( \sum_{k=1}^N  \Norm{  a_ku_k }{L^{p_1, p_2}}^{p_1}  \Big)^{1/p_1} \\
        &=  \Big( \sum_{k=1}^N a_k^{p_1} \Norm{  u_k }{L^{p_1, p_2}}^{p_1}  \Big)^{1/p_1}
        \lesssim  \Big( \sum_{k=1}^N a_k^{p_1}  \Big)^{1/p_1}
        \le \Big(\sum_{k=1}^N a_k^s\Big)^{1/s}\lesssim 1.
	\end{align*}
	On the other hand, if the condition $(ii)$ holds with $j=2$, then
    \begin{align*}
		\Norm{g^N}{L^{p_1,p_2}} 
        &=  \Big\| \Big(  \sum_{k=1}^N \|a_k u_k\|_{L^{p_2}}^{p_2} \Big)^{1/p_2} \Big\|_{L^{p_1}} \\
        &\le \Big\| \Big(  \sum_{k=1}^N \|a_k u_k\|_{L^{p_2}}^{s} \Big)^{1/s} \Big\|_{L^{p_1}} \\
        &= \Big\| \sum_{k=1}^N \|a_k u_k\|_{L^{p_2}}^{s} \Big\|_{L^{p_1/s}}^{1/s} \\
        &\le \Big( \sum_{k=1}^N \| \|a_ku_k\|_{L^{p_2}}^s \|_{L^{p_1/s}} \Big)^{1/s} \\
        &= \Big( \sum_{k=1}^N a_k^s \|u_k\|_{L^{p_1, p_2}}^s \Big)^{1/s}
        \lesssim  \Big(\sum_{k=1}^N a_k^s\Big)^{1/s}\lesssim 1.
    \end{align*}
    We used here, in addition to (ii), the condition (i), that $\ell^s \subset \ell^{p_2}$ as $s < p_2$ and
    that $L^{p_1/s}$ is a normed space as $s < p_1$.
	We have proved that in both cases
	\begin{align}\label{eq:boundUNI}
		\sup_{N}\Norm{g^N}{L^{p_1,p_2}}\lesssim 1.
	\end{align}

	Now, we set $h^N := \sum_{k=1}^Na_k\Phi$ and estimate 
	\begin{equation}\label{eq:dede}
		\begin{split}
			\bNorm{Ug^N - h^N}{L^{q_1,q_2}} 
            &= \BNorm{\sum_{k=1}^N a_k\left(Uu_k-\Phi\right) }{L^{q_1,q_2}} \\
            &\leq  \sum_{k=1}^N a_k \bNorm{Uu_k-\Phi}{L^{q_1,q_2}} \\
			&\leq \sum_{k=1}^Na_k2^{-k} \leq \Big(\sum_{k=1}^Na_k^s\Big)^{1/s}
            \Big(\sum_{k=1}^N2^{-ks'} \Big)^{1/s'} \lesssim 1.
		\end{split}
	\end{equation}
	By \eqref{eq:dede}, \eqref{eq:boundUNI} and the boundedness of $U$ (recalling
    also that $\Phi$ is non-trivial) we obtain a contradiction:
	\begin{align*}
        \sum_{k=1}^Na_k &\sim 	\sum_{k=1}^Na_k \Norm{\Phi}{L^{q_1,q_2}}
        = \bNorm{ h^N}{L^{q_1,q_2}}\\ 
                        &\leq \bNorm{Ug^N - h^N}{L^{q_1,q_2}}
                        + \bNorm{Ug^N}{L^{q_1,q_2}}
                        \lesssim 1+ \Norm{U}{L^{p_1,p_2}\to L^{q_1,q_2}}\lesssim 1.
	\end{align*}
	Indeed, this uniform bound contradicts the choice  $\{a_k\}\not\in \ell^1.$
    We are done.
\end{proof}

We now introduce the specific meaning of non-degeneracy and recall
a few results from \cite{AHLMO2021}.
\begin{defn}[Non-degenerate kernels]
Let $K_i$ be a Calder\'on-Zygmund kernel on $\R^{d_i}$.
We say that $K_i$ is {\em (symmetrically) non-degenerate} if there is a constant
$c_0 > 0$ such that for every $y_i\in\R^{d_i}$ and $r>0$ there exists
$x_i\in B(y_i,r)^c$ so that there holds
\[
    |K_i(x_i,y_i)|\ge \frac 1{c_0 r^{d_i}}
    \quad\Big(\text{and}\quad |K_i(y_i,x_i)|\ge \frac 1{c_0 r^{d_i}}\Big).
\]
\end{defn}
\begin{rem}
    Notice that the $x_i$ in the above definition necessarily satisfies $|x_i-y_i| \sim r$.
\end{rem}
If $T_i$ is a CZO whose kernel $K_i$ is a non-degenerate Calder\'on-Zygmund kernel,
then we say that $T_i$ is a non-degenerate CZO.
\begin{defn}[The ``reflected'' cube $\wt I_i$]
Let $K_i$ be a fixed non-degenerate Calder\'on-Zygmund kernel on $\R^{d_i}$,
and fix a large constant $A \ge 3$. For each cube $I_i \subset\R^{d_i}$
with center $c_{I_i}$ and sidelength $\ell(I_i)$,
we define another cube $\wt I_i\subset\R^{d_i}$ of the same size by choosing a center
$c_{\wt I_i}$, guaranteed by the non-degenaracy of $K_i$, so that
\begin{equation*}
  |c_{I_i}-c_{\wt I_i}| \sim A\ell(I_i),\qquad |K_i(c_{\wt I_i}, c_{I_i})| \sim (A \ell(I_i))^{-d_i}.
\end{equation*}
Notice that $\dist(I_i,\wt I_i)\sim A\ell(I_i)$.
If $K_i$ is symmetrically non-degenerate, we require in addition that
$|K_i(c_{I_i}, c_{\wt I_i})| \sim (A \ell(I_i))^{-d_i}$. While the choice of
$\wt I_i$ may not be unique in general, in the symmetric we arrange
$\,\wt{\!\wt I_i}=I_i$.
\end{defn}
If $K_1$ and $K_2$ are two (symmetrically) non-degenerate Calder\'on--Zygmund
kernels on $\R^{d_1}$ and $\R^{d_2}$, and $I_i \subset \R^{d_i}$
are cubes, we of course define $\wt I_i$ with respect to $K_i$. For each rectangle
$R=I_1\times I_2\subset\R^d = \R^{d_1} \times \R^{d_2}$, we define $\wt R:=\wt I_1\times \wt I_2$.

\begin{rem}\label{rem:symvsnon}
    In this paper, the nyances between symmetrically non-degenerate and just non-degenerate do not really
    manifest themselves. The issue is more prominent in \cite{AHLMO2021} -- in particular, a version
    of Lemma \ref{lem:absorb} with a weaker a priori assumption on the symbol
    requires a treatment taking into account such nyances (see \cite{AHLMO2021}).
    In this paper, mainly because of the work done in \cite{AHLMO2021} (as we soon discuss),
    we do not have to explicitly deal with such issues. What is important to us is that
    the boundedness table of the bi-commutator (proved in \cite{AHLMO2021}) from the Introduction is valid --
    this is true at least if both operators are symmetrically non-degenerate, but one could
    get away with somewhat weaker assumptions as well (see \cite{AHLMO2021}).
\end{rem}

Finally, we define the following notation
\[
L^{u_1,u_2}_{0,0}(I\times J) := \left\{ \phi \in L^{u_1,u_2}(I\times J):
\int_{I}\phi(x_1,\cdot)\ud x_1 = \int_{J}\phi(\cdot,x_2)\ud x_2 = 0  \right\}.
\]
The following rectangular weak factorization is from \cite{AHLMO2021}.
Recall that $\alpha_i$, a constant that appears in some of the estimates, is
just the kernel constant from the continuity estimate \eqref{eq:holcon}.
\begin{lem}
    \label{lem:awf} 
    Let $T1, T_2$ be non-degenerate CZOs on $\R^{d_i}$.
    Let $R=I\times J$ be a fixed rectangle and denote
    $$
    R_1 = \wt I\times J,\qquad R_2 =  I \times \wt J,\qquad R_3 = \wt R = \wt I \times \wt J.
    $$
    Then each $f\in L^{1,1}_{0,0}(R)$ can be expanded as 
    \begin{align*}
        f &= g_{\wt{R}} T_1T_2 h_R-T_2 h_R \cdot 
        T_1^*g_{\wt{R}}-T_1h_R
        \cdot T_2^*g_{\wt{R}}+h_R T_1^*T_2^*g_{\wt{R}}
        + \sum_{j=1}^3\tilde{f}_j,
    \end{align*}
    where the appearing functions satisfy:
    \begin{itemize}
        \item $g_{\wt{R}} = 1_{\wt{R}},$
        \item $h_R \in L^1(R)$ with $\abs{h_R(x)} \lesssim A^d \abs{f(x)},$
        \item $\wt{f}_j\in L_{0,0}^{1,1}(R_j)$, $j=1,2,3$,
        \item $	\abs{\tilde f_1(x)}
            \lesssim A^{-\alpha_1} 1_{\wt I}(x_1) \langle |f| \rangle_{I}(x_2),$
        \item $\abs{\tilde f_2(x)}
            \lesssim A^{-\alpha_2} \langle |f| \rangle_{J}(x_1)1_{\wt J}(x_2),$
        \item $\abs{\tilde f_3(x)}
            \lesssim A^{-\alpha_1} A^{-\alpha_2}  \langle |f| \rangle_{R}1_{\wt{R}}(x)$.
    \end{itemize}
    In particular, with $1 \le u_i \le \infty$, we have 
    $$
    \|h_R\|_{L^{u_1,u_2}} \lesssim A^d \|f\|_{L^{u_1,u_2}} 
    \qquad \textup{and} \qquad \Norm{\tilde{f}_j}{L^{u_1,u_2}}
    \lesssim A^{-\min(\alpha_1, \alpha_2)} \Norm{f}{L^{u_1,u_2}}.
    $$
    The implicit constants depend only on the dimension and the kernels.
\end{lem}

In what follows, our aim is to prove that if $[T_1, [b, T_2]]$ is compact, then
$b$ must be in one of our rectangular $\VMO$ spaces. We are proving an optimal
necessary condition (a condition that matches with our upper bounds) 
only in the range $p_1 \le q_1$, $p_2 \le q_2$, where, in addition, $p_1 \ne q_1$
or $p_2 \ne q_2$. In the diagonal, our proofs also yield the necessity of some
rectangular $\VMO$ condition (instead of the product $\VMO$ condition needed in the sufficiency
direction) -- see the Introduction for the full context. So even though the deep question
in the diagonal is the question about the necessity of the product $\VMO$, we find
that it is wortwhile to record a necessary condition in terms of the rectangular $\VMO$
as it is the best one can currently hope for. It also comes essentially for free from
our methods -- the only price being that we have to pay attention to the integrability
exponents used to define the rectangular $\VMO$ (discussed below).

So $p_i \le q_i$, $i = 1,2$, in what follows, as is the case in the rest of this paper.
Now, what is the right a priori condition on $b$ to run our proof? In \cite{AHLMO2021} we
went to some lengths to only
a priori assume $b \in L^1_{\loc}$, and then deduce just from this and the boundedness of the bi-commutator
that $b$ is in some other space (such as, a rectangular $\BMO$ space).
However, here we can simplify our life by utilizing the work done in \cite{AHLMO2021} --
if $[T_1, [b, T_2]]\colon L^{p_1, p_2} \to L^{q_1, q_2}$ is compact, it is also bounded, and so we can always
assume that $b$ is in the space given by the necessity results of \cite{AHLMO2021}
(see the table in the Introduction). That is to say, given our current restriction $p_i \le q_i$,
we can assume $b \in \BMO^{\beta_1, \beta_2}$, where, as usual
$$
  \beta_i := d_i\left( \frac{1}{p_i}-\frac{1}{q_i}\right).
$$
However, recall the technical detail that if $\beta_1 = \beta_2 = 0$ (i.e., $p_1 = q_1$ and $p_2 = q_2$
so we are at the diagonal), we do not have John--Nirenberg: the spaces $\BMO^{0, 0}_{v_1, v_2}$
and $\BMO^{0, 0}_{u_1, u_2}$ may differ from each other. In fact, the table in the Introduction
is loose about this fact -- only $\BMO_{\rect}$ is written there. However, \cite{AHLMO2021} actually
proved that boundedness $[T_1, [b, T_2]] \colon L^{p_1, p_2} \to L^{p_1, p_2}$ implies
$b \in \BMO^{0,0}_{p_1', p_2'}$ (and this is a better result than, say, only getting
$b \in \BMO^{0,0}_{1,1}$). So to take this diagonal case into consideration,
we formulate the following lemma with the a priori assumption $b \in \BMO^{\beta_1, \beta_2}_{p_1', p_2'}$
-- again, any exponents whatsoever can be used outside the diagonal case
(usually $v_1 = v_2 = 1$).

Finally, what is the non-degeneracy assumption needed on $T_1, T_2$ -- do they 
both need to be \textbf{symmetrically} non-degenerate? This was already discussed in Remark \ref{rem:symvsnon}.
Again, we will simply
assume that both of them are symmetrically non-degenerate -- then the boundedness table of the Introduction
can be used without any worries (but this is not strictly necessary, see \cite{AHLMO2021}). 

A version of the following lemma can be found in \cite{AHLMO2021} -- that version
is actually clearly stronger with better quantifications. However, compactness
is a qualitative property and the following version of this lemma is sufficient for
our use and has the benefit of being more \emph{ergonomic} (e.g., only one term on the RHS
of \eqref{eq:belowUNIF}). That is why we state and prove it in this particular form.
\begin{lem}\label{lem:absorb}   
    Let $T_1, T_2$ be symmetrically non-degenerate CZOs on $\R^{d_i}$.
    Let $p_i, q_i\in (1,\infty)$ with $p_i \le q_i$, $i=1,2$.
    Suppose that $b\in\BMO^{\beta_1,\beta_2}_{p_1',p_2'}$ --
    this is guaranteed by \cite{AHLMO2021} if $b \in L^1_{\loc}$ and $[T_1, [b, T_2]]$
    maps $L^{p_1, p_2} \to L^{q_1, q_2}$ boundedly.

    Let $R$ be an arbitrary rectangle. If $	\calO^{\beta_1,\beta_2}_{p_1',p_2'}(b, R) \geq c>0$,
    then, for a large enough parameter $A = A(c,\|b\|_{\BMO^{\beta_1,\beta_2}_{p_1',p_2'}})$
    (implicit in the definition of reflected cubes),
    there holds that 
    \begin{align}\label{eq:belowUNIF}
        \calO^{\beta_1, \beta_2}_{p_1', p_2'}(b, R) \lesssim
        \Babs{\Big \langle [T_1,[b,T_2]]\varphi_{R}, \psi_{\wt{R}} \Big\rangle},
    \end{align}
    where the test functions satisfy
    \begin{align}\label{eq:phis}
        \varphi_{R} = 1_{R}\varphi_{R},\qquad \bNorm{\varphi_{R} }{L^{p_1,p_2}}\lesssim_A 1,
    \end{align}
    \begin{align}\label{eq:psis}
        \psi_{\wt{R}} = 1_{\wt{R}}\varphi_{\wt{R}},\qquad \bNorm{\psi_{\wt{R}} }{L^{q_1',q_2'}}\lesssim 1.
    \end{align}
\end{lem}

\begin{proof}
    We note that the proof only explicitly requires that $T_1, T_2$ are non-degenerate
    and $b\in\BMO^{\beta_1,\beta_2}_{p_1',p_2'}$ -- it is not needed that $p_i \le q_i$
    or that the CZOs are \emph{symmetrically} non-degenerate (but then one cannot make the
    ``this is guaranteed by \cite{AHLMO2021}'' claim of the lemma and the result is of marginal
    interest in that form).

    By simple duality we write 
    \begin{align}
        \Norm{b-\langle b\rangle_{I} -\langle b\rangle_{J} 
        +\langle b\rangle_R}{L^{p_1', p_2'}(R)} \lesssim \Big| \int bf \Big|,
    \end{align}
    where $f\in L^{p_1,p_2}_{0,0}(R)$ with $ \|f\|_{L^{p_1,p_2}(R)} \le 1$.
    We apply Lemma  \ref{lem:awf} and write
    \begin{equation*}
        \begin{split}
            \pair{b}{f}
            &=\Bpair{b}{g_{\wt{R}} T_1T_2 h_{R}-T_2 h_{R} \cdot T_1^*g_{\wt{R}}-T_1h_{R}
            \cdot T_2^*g_{\wt{R}}+h_{R} T_1^*T_2^*g_{\wt{R}}} + \sum_{j=1}^3\pair{b}{\tilde{f}_j} \\
            &=\Bpair{b T_1T_2 h_{R}- T_1(b T_2 h_{R})-T_2(bT_1 h_{R})+ T_1T_2(bh_{R})}{g_{\wt{R}} }
            + \sum_{j=1}^3\pair{b}{\tilde{f}_j} \\
            &=-\Bpair{[T_1,[b,T_2]] h_{R}}{g_{\wt{R}} } + \sum_{j=1}^3\pair{b}{\tilde{f}_j}.
        \end{split}
    \end{equation*}

    In the following estimate we denote $R_j= I_j\times J_j$
    (recall the definitions of $R_1, R_2, R_3$ from Lemma \ref{lem:awf}).
    By $\wt{f}_j\in L_{0,0}^{1,1}(R_j)$ and H\"{o}lder's inequality
    \begin{align*}
        \babs{\pair{b}{\tilde{f}_j} }
        &= \Babs{\int_{I_j} \int_{J_j}
        (b- \ave{b}_{I_j}-\ave{b}_{J_j}+\ave{b}_{R_j})\wt{f}_j} \\
        &\leq \Norm{\wt{f}_j}{L^{p_1,p_2}} \Norm{b-\langle b\rangle_{I_j} -\langle b\rangle_{J_j} 
        +\langle b\rangle_{R_j}}{L^{p_1', p_2'}(R_j)} \\
        &\lesssim A^{-\min(\alpha_1, \alpha_2)}  \Norm{b}{\BMO^{\beta_1,\beta_2}_{p_1',p_2'}}
        |I|^{1/p_1'}|J|^{1/p_2'} \ell(I)^{\beta_1} \ell(J)^{\beta_2}.
    \end{align*}
    Thus, we obtain
    \begin{align*}
        \calO^{\beta_1,\beta_2}_{p_1',p_2'}(b,R) 
        &\lesssim \ell(I)^{-\beta_1} \ell(J)^{-\beta_2} \frac{|\langle b, f\rangle|}{|I|^{1/p_1'} |J|^{1/p_2'}}\\ 
        &\lesssim  \Babs{\Bpair{[T_1,[b,T_2]] \varphi_R}{\psi_{\wt{R}}}}
        + A^{-\min(\alpha_1, \alpha_2)} \Norm{b}{\BMO^{\beta_1,\beta_2}_{p_1',p_2'}},
    \end{align*}
    where we have denoted 
    \[
        \varphi_R := h_R\qquad \psi_{\wt{R}} := g_{\wt{R}}|I|^{-\frac{1}{q_1'}}|J|^{-\frac{1}{q_2'}}.
    \]
    Here we used that
    $$
        \ell(I)^{-\beta_1} |I|^{-1/p_1'} = |I|^{-\beta_1/d_1 - 1/p_1'}
        = |I|^{-1/p_1 + 1/q_1 -1/p_1'} = |I|^{1/q_1 - 1} = |I|^{-1/q_1'}
    $$
    and similarly for the $J$ term. From Lemma \ref{lem:awf} we know that
    $$
        \|\varphi_R\|_{L^{p_1, p_2}} \lesssim A^d 
    $$
    and
    $$
        \|\psi_{\wt R}\|_{L^{q_1', q_2'}} = 1
    $$
    so these functions satisfy \eqref{eq:phis} and \eqref{eq:psis}.
    Finally, using the assumption $\calO^{\beta_1,\beta_2}_{p_1',p_2'}(b,R) \geq c$,
    we get \eqref{eq:belowUNIF} provided $A$ is sufficiently large.
\end{proof}

We recall the following standard result about compact operators that is used  in the upcoming proof.
\begin{lem}
    Let $X,Y$ be Banach spaces and $T:X\to Y$ be a bounded linear operator.
    Then, $T\in\calK(X,Y)$ if and only if $T^*\in\calK(X^*,Y^*).$
\end{lem}

Finally, we are ready to prove our main necessity theorem.
\begin{thm}\label{thm:necessity}
    Let $T_1, T_2$ be symmetrically non-degenerate CZOs on $\R^{d_i}$.
    Let $p_i, q_i\in (1,\infty)$ with $p_i \le q_i$, $i=1,2$.
    If $[T_1,[b,T_2]]\in \calK(L^{p_1,p_2},L^{q_1,q_2})$,
    then $b\in\VMO^{\beta_1,\beta_2}_{p_1',p_2'}.$
\end{thm}

\begin{proof}
    Similarly as in Lemma \ref{lem:absorb}, the proof only formally requires that
    $T_1, T_2$ are non-degenerate, the commutator is compact
    and $b \in \BMO^{\beta_1,\beta_2}_{p_1',p_2'}$. It is not strictly speaking needed that $p_i \le q_i$
    or that the CZOs are \emph{symmetrically} non-degenerate, but then
    $b \in \BMO^{\beta_1,\beta_2}_{p_1',p_2'}$ is an assumption, not a corollary
    of \cite{AHLMO2021} like under the current assumptions of the theorem.

	Aiming for a contradiction, assume that $b\not \in\VMO^{\beta_1,\beta_2}_{p_1',p_2'}$.
    Recall Lemma \ref{lem:seq} -- our aim is to construct a sequence of functions $(u_k)$
    like there to reach a contradiction. The underlying operator $U$ will
    be $C_b := [T_1, [b, T_2]] \colon L^{p_1, p_2} \to L^{q_1, q_2}$
    or $C_b^* = [T_1^*, [b, T_2^*]] \colon L^{q_1', q_2'} \to L^{p_1', p_2'}$ -- both
    of them are compact, in particular, bounded.

    We will soon use the fact that $b\not \in\VMO^{\beta_1,\beta_2}_{p_1',p_2'}$ to
    construct a sequence of rectangles $R_j = I_j \times J_j$ so that for some $c>0$
    we have both of the following items holding.
    \begin{enumerate}
        \item We have $\calO^{\beta_1,\beta_2}_{p_1', p_2'}(b,R_j) \ge c$ for all $j$.
        \item One of the four collections $\{I_j\}$, $\{J_j\}$, $\{\wt I_j\}$
            or $\{\wt J_j\}$ is disjoint.
    \end{enumerate}
    We first show, however, that given such a collection $(R_j)$ we can reach a contradiction.
    This step uses the compactness of $C_b$ or $C_b^*$. First, to each $R_j$ apply Lemma
    \ref{lem:absorb} -- this gives functions $\varphi_j := \varphi_{R_j}$ and
    $\psi_j := \psi_{\wt R_j}$ so that
    \begin{itemize}
        \item $\calO^{\beta_1, \beta_2}_{p_1', p_2'}(b, R_j) \lesssim 
            \babs{\big \langle C_b\varphi_j, \psi_j \big\rangle}$,
        \item $\supp \varphi_j \subset R_j$, $\supp \psi_j \subset \wt R_j$,
        \item $\|\varphi_j\|_{L^{p_1, p_2}} \lesssim 1$ and $\|\psi_j\|_{L^{q_1', q_2'}}
            \lesssim 1$.
    \end{itemize}
    If either the collection $\{I_j\}$ or $\{J_j\}$ is disjoint, we use the compactness
    or $C_b$ to say that, after passing to a subsequence if necessary, $C_b \varphi_j \to
    \Phi$ for some $\Phi \in L^{q_1, q_2}$. Set $u_j := \varphi_j$ and $U := C_b$.
    Clearly, $(i)$ and $(ii)$ of Lemma \ref{lem:seq} hold for $(u_j)$, and also $(iii)$ holds except for the fact
    that we have not yet checked that $\Phi$ is non-trivial. But this also follows, since we have
    $$
        c < \calO^{\beta_1, \beta_2}_{p_1', p_2'}(b, R_j) \lesssim 
            \babs{\big \langle C_b\varphi_j, \psi_j \big\rangle}
            \lesssim \|C_b \varphi_j\|_{L^{q_1, q_2}}
    $$
    for all $j$. This is a contradiction by Lemma \ref{lem:seq}. Now, if instead
    $\{\wt I_j\}$ or $\{\wt J_j\}$ is disjoint, use the compactness of $C_b^*$
    to conclude that, after passing to a subsequence if necessary, $C_b^* \psi_j \to
    \Phi$ for some $\Phi \in L^{p_1', p_2'}$. Set $u_j := \psi_j$ and $U := C_b^*$.
    Clearly, $(i)$ and $(ii)$ of Lemma \ref{lem:seq} hold for $(u_j)$, and also $(iii)$ holds except for the fact
    that we have not yet checked that $\Phi$ is non-trivial. But this also follows, since we have
    $$
        c < \calO^{\beta_1, \beta_2}_{p_1', p_2'}(b, R_j) \lesssim 
            \babs{\big \langle C_b\varphi_j, \psi_j \big\rangle}
            = \babs{\big \langle \varphi_j, C_b^*\psi_j \big\rangle}
            \lesssim \|C_b^* \psi_j\|_{L^{p_1', p_2'}}
    $$
    for all $j$. This is a contradiction by Lemma \ref{lem:seq}.

    So all that remains to do is to contruct such a collection of rectangles using
    that $b\not \in\VMO^{\beta_1,\beta_2}_{p_1',p_2'}$. First, the failure of the
    $\VMO^{\beta_1,\beta_2}_{p_1',p_2'}$ condition just means that there exists
    $c > 0$ and a sequence of rectangle $R_j = I_j \times J_j$ so that
    $\calO^{\beta_1,\beta_2}_{p_1', p_2'}(b,R_j) \ge c$ for all $j$, 
	and at least one of the following conditions holds:
	\begin{enumerate}
		\item $\ell(I_j)\to 0$ as $j\to\infty,$
		\item $\ell(I_j)\to \infty$ as $j\to\infty,$
		\item $\dist(I_j,0)\to\infty$ as $j\to\infty,$
	\end{enumerate}
    or one of the symmetrical conditions with the cubes $I_j$ replaced
    by the cubes $J_j$ holds. By symmetry, it is enough to consider only one of the cases
    $(1),(2)$ and $(3)$ and show that we can always choose a subsequence of the cubes $I_j$
    so that either the said subsequence or the corresponding subsequence of reflected
    cubes forms a disjoint collection of cubes.
	
	\subsubsection*{Suppose that $(3)$ holds.}
	We let $i(1) = 1$. Suppose that we have picked $i(1) < \cdots < i(N)$
    so that $\{I_{i(j)}\}_{j=1}^N$ is pairwise disjoint.
	Let $M>0$ be such that $\cup_{j=1}^N I_{i(j)}\subset B(0,M)$.
	By $(3)$ there exists $i(N+1) > i(N)$ so that
    $\dist(I_{i(N+1)},0) > M$. Now, obviously
    $\{I_{i(j)}\}_{j=1}^{N+1}$ is pairwise disjoint and so, by induction, we are done with the case (3). 
	
	\subsubsection*{Suppose that $(2)$ holds.}
    Since we have already dealt with the case $(3)$, we can assume that $(3)$ does not hold
    (otherwise we are done).
    Using $\dist(I_j,\wt{I_j})\sim \ell(I_j)$ and $(2)$
    we deduce that 
    \begin{align*}
        \lim_{j\to\infty}\dist(\wt{I}_j,0) = \infty.
    \end{align*}
    But then we are in the case $(3)$ for the reflected sequence $\{\wt{I}_j\}$
    and we can pick a disjoint subsequence out of them.
	
	\subsubsection*{Suppose that $(1)$ holds.}
	We may again assume that the condition $(3)$ does not hold. 
    Let now $M>0$ be such that $\cup_{j=1}^{\infty} I_j \subset [-M, M]^{d_1} =: Q_0$.
	Let $\calD_1(Q_0)$ consist of the $2^{d_1}$ closed subcubes of $Q_0$ 
    obtained by bisecting the sides (the closed dyadic children of $Q_0$).
    There must exist $Q_1\in \calD_1(Q_0)$ so that 
	$F(Q_1) :=	\{ j\geq 1 : I_j\cap Q_1\not=\emptyset \}$
	has infinitely many elements. Similarly, there must exist $Q_2 \in \calD_1(Q_1) \subset \calD_2(Q_0)$
    so that $F(Q_2)$ has infinitely many elements. Continue like this to obtain closed
    cubes $Q_0 \supsetneq Q_1 \supsetneq Q_2 \supsetneq \cdots$ with $\ell(Q_j) \to 0$,
    and the corresponding infinite index sets $F(Q_j) \supset F(Q_{j+1})$.
    Pick $i(1) \in F(Q_1)$. Then, using that $F(Q_2)$ is infinite, pick
    $i(2) \in F(Q_2)$ with $i(2) > i(1)$. Continue like this to obtain
    indices $i(1) < i(2) < \cdots$ so that $i(j) \in F(Q_j)$, i.e.,
    $I_{i(j)} \cap Q_j \ne \emptyset$. We then re-index the situation so that $(I_j)_j$
    now denotes this subsequence $(I_{i(j)})_j$, so now $I_j \cap Q_j \ne \emptyset$ for all $j$.
    Let
    $$
        \{x\} = \bigcap_j Q_j;
    $$
    the intersection is non-empty as each $Q_j \supset Q_{j+1}$ is compact, and then clearly
    a singleton as $\ell(Q_j) \to 0$. Let
    $$
        \calF := \{j \colon x \not \in I_j\}.
    $$
    Then either $\calF$ or $\N \setminus \calF$ is infinite. But notice that
    $\N \setminus \calF \subset \{j\colon x \not \in \wt I_j\}$ as $I_j$ and
    $\wt I_j$ are disjoint -- these are symmetric situations leading us to
    discover either a subsequence of disjoint cubes $I_j$ or disjoint cubes $\wt I_j$.
    So without loss of generality suppose, for instance, that $\calF$ is infinite.
    Now, we use that $I_j \cap Q_j \ne \emptyset$ implies that
    $$
        \dist(I_j, x) \le \diam(Q_j) \to 0.
    $$
    In addition, by $(1)$ we have that $\ell(I_j) \to 0$.
    Pick $i(1) \in \calF$. Suppose $i(1) < i(2) < \cdots < i(N)$, $i(j) \in \calF$, have
    already been chosen so that the cubes $I_{i(1)}, \ldots, I_{i(N)}$ are disjoint.
    Then pick -- using that $\calF$ is infinite -- $L \in \calF$, $L > i(N)$, so that
    $$
    \dist(I_L, x) + \diam(I_L) < \dist\Big(\bigcup_{j=1}^N I_{i(j)}, x \Big)/10, 
    $$
    where it is important to notice that the RHS is strictly positive as $x$
    does not belong to any of these finitely many cubes. Set $i(N+1) := L$.
    It is now clear that $(I_{i(j)})_{j=1}^{N+1}$ is a disjoint collection of cubes and $i(j) \in \calF$
    for all $j = 1, \ldots, N+1$. Induction finishes the proof.
\end{proof}

\subsection{The cases where $\beta_i\geq 1$ for $i=1$ or $i=2$} 
We end by verifying that if $\beta_1 \ge 1$ or $\beta_2 \ge 1$, compactness
implies that the bi-commutator vanishes. This was used in the sufficiency proofs
to rule out having to have to do something in these cases.

So assume that $[T_1,[b,T_2]]:L^{p_1,p_2}\to L^{q_1,q_2}$ is compact,
$b \in \BMO^{\beta_1, \beta_2}$ (as previously, see the beginning of the proof
of Theorem \ref{thm:necessity}, this follows if $p_i \le q_i$ by \cite{AHLMO2021}),
and $\beta_i\geq 1$ for some $i=1,2$ (we do not have to specify the exponents in the $\BMO$
norm as we cannot be in the diagonal). We may as well assume $\beta_1 \ge 1$ (the proof below
works equally well when $\beta_2 \ge 1$).
Now, Theorem \ref{thm:necessity} (or rather its proof if it is not the case that $p_i \le q_i$)
gives us that $b\in \VMO^{\vec{\beta}}$.
Next, by Lemma \ref{lem:VMOa0uniform}, we have the embedding
\[
\VMO^{\vec{\beta}}(\R^d) \subset  \VC^{0,\beta_1}(\R^{d_1}, \BMO^{\beta_2}(\R^{d_2})).
\]
We check that $b\in \VMO^{\vec{\beta}}(\R^d)$, via the above embedding,
implies that $b$ is a sum of two functions each of which is constant in either the
$x_1$ or the $x_2$ variable; since the bi-commutator annihilates such functions, we are then done.

Fix $\varepsilon > 0$.
Choose $t \in (0,1)$ so that $|z_1-w_1| < t$ implies that 
$$
\|b(z_1, \cdot) - b(w_1, \cdot)\|_{\BMO^{\beta_2}} \le \varepsilon|z_1-w_1|^{\beta_2}.
$$
Fix arbitrary $x_1,y_1\in \R^{d_1}$ and write the line segment $[x_1, y_1]$ (unless
we already have $|x_1 - y_1| < t$)
connecting $x_1$ and $y_1$ as a union $[x_1, y_1]= \bigcup_{j=1}^N L_j$
with $L_j = [z_{j-1}, z_j]$ (where $z_0 = x_1$ and $z_N = y_1$) satisfying
$|z_j - z_{j-1}| = t/2$ for all $j=1,\ldots, N-1$ and $|z_N - z_{N-1}| \le t/2$.
We now have
\begin{align*}\|b(x_1,\cdot) - b(y_1,\cdot) \|_{\BMO^{\beta_2}} 
    &\le  \sum_{j=1}^{N}  \|b(z_j,\cdot) - b(z_{j-1},\cdot) \|_{\BMO^{\beta_2}} \\
    &\le \varepsilon \sum_{j=1}^N t^{\beta_2} \le \varepsilon Nt \lesssim \varepsilon |x_1-y_1|.
\end{align*}
As $x_1, y_1$ and $\varepsilon > 0$ were arbitrary, this shows that
\[
    \|b(x_1,\cdot) - b(y_1,\cdot) \|_{\BMO^{\beta_2}}  = 0
\]
for all $x_1, y_1$.
Thus, we have that $b(x_1,\cdot) - b(y_1,\cdot) = C(x_1,y_1)$ a.e. in $\R^{d_2}$,
and taking $y_1 = 0,$ we obtain that given $x_1$ we have 
\[
b(x_1,x_2) = b(0,x_2) +  C(x_1,0)
\]
for almost every $x_2$.

\bibliography{references}
\end{document}